\documentclass[11pt]{article}

 \usepackage{amsmath,amssymb,amscd,amsthm,esint}
 
\usepackage{graphics,amsmath,amssymb,amsthm,mathrsfs}

\usepackage{graphics,amsmath,amssymb,amsthm,mathrsfs,amsfonts}

\usepackage{sidecap}
\usepackage{float}
\usepackage{extarrows}
\usepackage{booktabs}
\usepackage{verbatim}
\usepackage{hyperref}
\usepackage[usenames,dvipsnames]{xcolor}

\newtheorem{thm}{Theorem}[section]
\newtheorem{lemma}[thm]{Lemma}
\newtheorem{cor}[thm]{Corollary}

\theoremstyle{definition}

\def\XXint#1#2#3{{\setbox0=\hbox{$#1{#2#3}{\int}$}
         \vcenter{\hbox{$#2#3$}}\kern-.5\wd0}}

\def\R{\mathbb{R}}

\def\e{\varepsilon}

\def\loc{\text{loc}}

\numberwithin{equation}{section}

\begin{document}

\title{Darcy's Law for Porous Media \\ with Multiple Microstructures}

\author{
Zhongwei Shen\thanks{Supported in part by NSF grant DMS-2153585.}}
\date{}

\maketitle

\begin{abstract}

In this paper we study the homogenization of the Dirichlet problem for the Stokes equations
in a  perforated domain with multiple microstructures. First, under the assumption that the interface between subdomains
is a union of Lipschitz surfaces, 
we show that the effective velocity and pressure are governed by a Darcy law, where the permeability
matrix is piecewise constant.
The key step is to prove that the effective pressure is continuous across the interface, using Tartar's method of test functions.
Secondly, we establish the sharp error estimates for the convergence of the velocity and pressure, assuming the interface satisfies
certain smoothness and geometric conditions.
This is achieved by constructing two correctors.
One of them is used to correct the discontinuity of the two-scale approximation on the interface,  while
the other is used to correct the discrepancy between  boundary values of the solution and its approximation.

\medskip


\medskip

\noindent {\it MR (2020) Subject Classification}: 35Q35; 35B27; 76D07.

\end{abstract}

\section{Introduction}

In this paper we study the homogenization of the Dirichlet problem 
for the Stokes equations in a perforated domain $\Omega_\e$,
  \begin{equation}\label{Stokes-1}
  \left\{
  \aligned
  -\e^2 \mu  \Delta u_\e +\nabla p_\e & = f & \quad & \text{ in } \Omega_\e,\\
  \text{\rm div} (u_\e) & =0 & \quad & \text{ in } \Omega_\e,\\
  u_\e & = 0 & \quad & \text{ on }  \partial \Omega_\e,
  \endaligned
  \right.
    \end{equation}
    where $0< \e< 1$ and $\mu>0$ is a constant.
  Let $\Omega$ be a bounded Lipschitz 
  domain in $\R^d$, $d\ge 2$.
   Let  $\{ \Omega^\ell: 1\le \ell\le L\}$  be a finite number of disjoint subdomains of $\Omega$, each with a
  Lipschitz boundary, such that 
  \begin{equation}\label{domain-1}
  \overline{\Omega} =  \bigcup_{\ell =1}^L \overline{\Omega^\ell }.
  \end{equation}
  To describe the porous domain $\Omega_\e$,  let $Y=[0, 1]^d$ be a closed unit cube and $\{ Y_s^\ell: 
  1\le \ell  \le L \}$
    open subsets (solid parts) of $Y$ with Lipschitz boundaries.
  Assume that  for $1\le \ell \le L$, 
  dist$ (\partial Y, \partial Y_s^\ell )>0$ and 
  $Y^\ell_f=Y \setminus \overline{Y_s^\ell}$ (the fluid part) is connected. 
  For $ 0<  \e <1$ and $1\le \ell \le L$, define
  \begin{equation}\label{subdomain}
  \Omega_\e^\ell
  =\Omega^\ell \setminus \bigcup_z \e (\overline{Y_s^\ell} +z),
  \end{equation}
  where $z \in \mathbb{Z}^d$ and the union is taken over those $z$'s for which $\e (Y+z)\subset \Omega^\ell$.
  Thus the subdomain $\Omega^\ell$ is perforated periodically, using the solid obstacle  $Y_s^\ell$.
  Let
  \begin{equation}\label{domain-e}
  \Omega_\e = \Sigma \cup  \bigcup_{\ell=1}^L \Omega_\e^\ell
  =\Omega\setminus \bigcup_{\ell=1}^L \bigcup_z \e (\overline{Y_s^\ell} +z),
  \end{equation}
  where $\Sigma$ is the interface between subdomains,  given by
  \begin{equation}\label{interface}
  \Sigma= \Omega\setminus \bigcup_{\ell=1}^L\Omega^\ell
  =\bigcup_{\ell=1}^L \partial \Omega^\ell \setminus \partial \Omega.
  \end{equation}
  
  For  $f\in L^2(\Omega; \R^d)$,
  let $(u_\e, p_\e) \in H_0^1(\Omega_\e; \R^d)\times L^2(\Omega_\e)$ be the weak solution of \eqref{Stokes-1}
  with $\int_{\Omega_\e} p_\e\, dx =0$.
  We extend $u_\e$ to the whole domain $\Omega$ by zero.
  Let $P_\e$ denote the extension of $p_\e$ to $\Omega$, defined by \eqref{ext-1}.
  In the case $L=1$, where $\Omega$ is perforated periodically with small holes of same shape,
  it is well known  that as $\e\to 0$,
  $u_\e \to u_0$ weakly in $L^2(\Omega; \R^d)$ and $P_\e \to P_0$ strongly in $L^2(\Omega)$, where
the effective velocity and pressure  $(u_0, P_0)$ are governed by the Darcy law,
  \begin{equation}\label{Darcy-1}
  \left\{
  \aligned
  u_0 & =\mu^{-1} K (f-\nabla P_0) & \quad &  \text{ in } \Omega, \\
  \text{div} (u_0) & =0 & \quad & \text{ in } \Omega,\\
  u_0\cdot n & =0 & \quad & \text{ on } \partial \Omega,
  \endaligned
  \right.
  \end{equation}
  with $\int_\Omega P_0\, dx=0$.
  Note that in \eqref{Stokes-1} we have normalized the velocity vector by a factor $\e^2$, where $\e$ is the period.
  For references on the Darcy law we refer to the reader to \cite{Sanchez-1980, LA-1990,
  Allaire-89,  Allaire-91a, Allaire-1997}.
 
  In \eqref{Darcy-1}
  the permeability matrix $K$ is a $d\times d$ positive definite, constant and symmetric matrix and
  $n$ denotes the outward unit normal to $\partial\Omega$.
  It  was observed in \cite{Allaire-91a} by G. Allaire that 
as $\e \to 0$,
\begin{equation}
u_\e - \mu^{-1} W(x/\e) (f-\nabla P_0) \to 0 \quad \text{strongly in } L^2(\Omega; \R^d),
\end{equation}
where $W(y)$ is an 1-periodic $d\times d$ matrix defined by a cell problem and
$\fint_Y W(y)\, dy =K$.
Recently, it was proved in \cite{Shen-Darcy-1} by the present author that
\begin{equation}\label{shen-1}
\| u_\e - \mu^{-1} W(x/\e) (f-\nabla P_0)\|_{L^2(\Omega)}
+ \| P_\e -P_0\|_{L^2(\Omega)}
\le C\sqrt{\e} \| f\|_{C^{1, 1/2}(\Omega)},
\end{equation}
and that
\begin{equation}\label{shen-2}
\| \e \nabla u_\e
-\mu^{-1} \nabla W(x/\e) (f-\nabla P_0) \|_{L^2(\Omega)}
\le C \sqrt{\e} \| f\|_{C^{1, 1/2}(\Omega)}.
\end{equation}
We point out that 
due to the discrepancy between boundary values of  $\mu^{-1} W(x/\e)(f-\nabla P_0)$  and $u_\e$
on $\partial \Omega$,
 the $O(\e^{1/2}) $ convergence rates in \eqref{shen-1}-\eqref{shen-2}
are sharp.
See \cite{Mik-1996} for an earlier  partial result on solutions with periodic boundary conditions.
  
The primary purpose of this paper is to study the  Darcy law for the case $L\ge 2$, 
where the domain $\Omega$ is divided into 
several  subdomains and different  subdomains are perforated with small holes of different shapes.
 
  \begin{thm}\label{main-thm-1}
  Let $\Omega$ be a bounded Lipschitz domain in $\R^d$, $d\ge 2$,  and $\Omega_\e$ be given by \eqref{domain-e}.
  Let  $(u_\e, p_\e) \in H_0^1(\Omega_\e; \R^d) \times L^2(\Omega_\e)$ be a weak solution of 
  \eqref{Stokes-1}, where $f\in L^2(\Omega; \R^d)$ and $\int_{\Omega_\e} p_\e\, dx =0$.
  Let $P_\e$ be the extension of $p_\e$, defined by \eqref{ext-1}.
  Then $u_\e \to u_0$ weakly in $L^2(\Omega; \R^d)$ and $P_\e-\fint_\Omega P_\e \to P_0$ strongly in $L^2(\Omega)$, as $\e \to 0$, 
  where $P_0\in H^1(\Omega)$ and 
 $(u_0, P_0)$ is governed by the Darcy law \eqref{Darcy-1} with
 the matrix
 \begin{equation}\label{K-0}
 K =\sum_{\ell =1}^L K^\ell \chi_{\Omega^\ell} \qquad \text{ in } \Omega.
 \end{equation} 
  \end{thm}
  
  The matrix $K^\ell$ in \eqref{K-0}  is the (constant) permeability matrix associated 
  with the solid obstacle  $Y_s^\ell$.
  Thus, the matrix $K$ is piecewise constant in $\Omega$, taking value $K^\ell $ in the subdomain $\Omega^\ell$,
  and
  \begin{equation}\label{eff-1}
  u_0= K^\ell (f-\nabla P_0) \quad \text{ in } \Omega^\ell.
  \end{equation}
  Since $\text{\rm div}(u_0)=0$ in $\Omega$ and $P_0\in H^1(\Omega)$, both the normal component 
  $u_0\cdot n$  and  
  $P_0$ are continuous across the interface $\Sigma$ (in the sense of trace) between subdomains.
  However, the tangential components of $u_0$ may not be continuous across $\Sigma$.
  
  The Dirichlet problem for the Stokes equations  \eqref{Stokes-1} is used to model fluid flows in porous media with different microstructures in 
  different subdomains.
  The continuity of the effective pressure $P_0$
  and the normal component  $u_0\cdot n$ of the effective velocity across the interface
  is generally accepted in engineering  \cite{Dagan, Mikelic-two}.
  Theorem \ref{main-thm-1} is probably known to experts.
  However, to the best of the author's knowledge,
   the existing  literatures on rigorous proofs only treat the case of flat interfaces.
  In particular, the result was proved    in \cite{Mikelic-two}  under the assumption that $d=2$,
  the interface $\Gamma =\R \times \{0 \}$ and the solutions are 1-periodic in the direction $x_1$. 
  Also see related work in   \cite{B, MGG}.
   We provide a proof here for the general case, where the interface is a union of Lipschitz surfaces, using Tartar's method of test functions.
  We point out that the proof for \eqref{eff-1}  and  $P_0\in H^1(\Omega^\ell)$ for each $\ell$
  is the same as in the classical case $L=1$.
  The challenge is to show that the effective pressure $P_0$ is continuous across the interface and thus $P_0\in H^1(\Omega)$, 
  which is essential for proving  the uniqueness of the limits of subsequences of $\{ u_\e\}$.
  
  Our main contribution in this paper is on the sharp convergence rates and error estimates  for $u_\e$ and $P_\e$. 
  We are able to extend the results in \cite{Shen-Darcy-2} for  the case $L=1$ to the case $L\ge 2$ under 
  some smoothness and geometric conditions on subdomains.
  More specifically, 
we assume that  each subdomain is a bounded $C^{2, 1/2}$ domain, and that 
  there exists $r_0>0$ such that if $x_0\in \partial \Omega^k \cap \partial \Omega^m$ for some $1\le k, m \le L$ and $k \neq m$, there exists 
  a coordinate system, obtained from the standard one by translation and rotation, such that 
  \begin{equation}\label{condition-1}
  \aligned
  B(x_0, r_0)\cap \Omega^k
   & =B(x_0, r_0) \cap \big\{ (x^\prime, x_d)\in \R^d: x_d > \psi (x^\prime) \big\},\\
   B(x_0, r_0)\cap \Omega^m
   & =B(x_0, r_0) \cap \big\{ (x^\prime, x_d)\in \R^d: x_d <  \psi (x^\prime) \big\},
   \endaligned
  \end{equation}
  where $\psi: \R^{d-1} \to \R$ is a $C^{2, 1/2}$ function.
  Roughly speaking, this means that inside a small ball centered on the interface $\Sigma$, the domain $\Omega$ is 
  divided by $\Sigma$  into  exactly two subdomains.
  In particular, the condition  excludes the cases where the interface intersects with each other or with the boundary of $\Omega$.
  
  The following is the main result of the paper.
  The matrix $W^\ell (y)$ in \eqref{main-1}-\eqref{main-2} is the 1-periodic  matrix 
  associated with the solid obstacle $Y_s^\ell$. 
  
  \begin{thm}\label{main-thm-2}
  Let $\Omega$ be a bounded $C^{2, 1/2}$ domain and $\Omega_\e$ be given by \eqref{domain-e}.
  Assume that the subdomains 
  $\{ \Omega^\ell \} $ are  bounded $C^{2, 1/2}$ domains satisfying the condition \eqref{condition-1}.
  Let $(u_\e, P_\e)$ and $(u_0, P_0)$  be the same as in Theorem \ref{main-thm-1}.
  Then, for $f\in C^{1, 1/2}(\Omega; \R^d)$,  
  \begin{equation}\label{main-1}
  \sum_{\ell=1}^L
  \| u_\e - \mu^{-1} W^\ell (x/\e) (f-\nabla P_0)\|_{L^2(\Omega^\ell) }
  + \| P_\e -\fint_\Omega P_\e -P_0\|_{L^2(\Omega)}
  \le C \sqrt{\e} \| f\|_{C^{1, 1/2} (\Omega)},
  \end{equation}
  and
  \begin{equation}\label{main-2}
  \sum_{\ell=1}^L
  \| \e \nabla u_\e -\mu^{-1} \nabla W^\ell (x/\e) (f-\nabla P_0)\|_{L^2(\Omega^\ell)}
  \le 
  C \sqrt{\e} \| f\|_{C^{1, 1/2} (\Omega)},
  \end{equation}
where $C$ depends on $d$, $\mu$, $\Omega$, $\{\Omega^\ell\}$ and $\{Y_s^\ell\}$.
  \end{thm}
  
  As we mentioned earlier, the sharp convergence rates in  \eqref{main-1} and \eqref{main-2} were proved 
  in \cite{Shen-Darcy-1} for the case $L=1$. 
  In the case of two porous media with a flat interface,
  partial results were obtained  in \cite{Mikelic-two} for solutions with periodic boundary conditions.
  Theorem \ref{main-thm-2} is  the first result that treats the general case of smooth interfaces.
  
  As in \cite{Mikelic-two}, the basic idea in our approach to Theorem \ref{main-thm-2} is to use
  \begin{equation}\label{V}
  V_\e (x)=\sum_{\ell=1}^L W^\ell (x/\e) (f-\nabla P_0) \chi_{\Omega_\e^\ell}
  \end{equation}
  to approximate the solution $u_\e$ and obtain the error estimates  by the energy method.
  Observe that $V_\e=0$ on $\Gamma_\e=\partial \Omega_\e\setminus \partial \Omega$.
  There are three main issues with this approach: (1) the divergence of $V_\e$ is not small in $L^2$;
  (2) $V_\e$ does not agree with $u_\e$ on $\partial \Omega$; and (3) $V_\e$ is not in $H^1(\Omega_\e; \R^d)$,  as 
  it is not continuous across the interface.
  To overcome these difficulties, we introduce three corresponding 
  correctors: $\Phi_{\e}^{(1)}$, $\Phi_{\e}^{(2)}$, and $\Phi_{\e}^{(3)}$.
  To correct the  divergence of $V_\e$, we construct $\Phi_\e^{(1)} \in H_0^1(\Omega_\e; \R^d)$ with the property that
  \begin{equation}\label{p1}
  \e \|\nabla \Phi_\e^{(1)} \|_{L^2(\Omega_\e^\ell)}
  + \| \text{\rm div} (\Phi_\e^{(1)} +V_\e) \|_{L^2(\Omega_\e^\ell)}
  \le C \sqrt{\e} \| f\|_{C^{1, 1/2}(\Omega)}
  \end{equation}
  for $1\le \ell \le L$.
  The construction of $\Phi_\e^{(1)}$  is similar to that in \cite{Mik-1996, Mikelic-two, Shen-Darcy-1}.
  Next, we correct the boundary data of $V_\e$ on $\partial \Omega$ by constructing $\Phi_\e^{(2)} \in H^1(\Omega_\e; \R^d)$
  such that $\Phi_\e^{(2)} + V_\e =0$ on $\partial\Omega$, $\Phi_\e^{(2)} =0$ on $\Gamma_\e$, and
  that 
  \begin{equation}\label{p2}
   \e \|\nabla \Phi_\e^{(2)} \|_{L^2(\Omega_\e)}
  + \| \text{\rm div} (\Phi_\e^{(2)} ) \|_{L^2(\Omega_\e)}
  \le C \sqrt{\e} \| f\|_{C^{1, 1/2}(\Omega)}.
  \end{equation}
The construction of  $\Phi_\e^{(2)}$ is similar to that in \cite{Shen-Darcy-1} for the case $L=1$.
The key observation is that  the normal component of  $V_\e$ on $\partial \Omega$ can be written in the form
\begin{equation}\label{tan-1}
\e \nabla_{\tan} \left( \phi(x/\e)\right) \cdot g,
\end{equation}
where $\nabla_{\tan}$ denotes the tangential gradient on $\partial\Omega$.
We remark that a similar observation is also used in the proof of Theorem \ref{main-thm-1}.
Finally, to correct the discontinuity of $V_\e$ across  the interface, we introduce
\begin{equation}\label{p3}
\Phi_\e^{(3)} =\sum_{\ell=1}^L I_\e^\ell (x) (f-\nabla P_0)\chi_{\Omega_\e^\ell},
\end{equation}
with the properties that $ V+ \Phi_\e^{(3)} \in H^1(\Omega_\e; \R^d)$, $\Phi_\e^{(3)} =0$ on $\partial \Omega_\e$, and that
\begin{equation}\label{p4}
   \e \|\nabla \Phi_\e^{(3)} \|_{L^2(\Omega_\e^\ell)}
  + \| \text{\rm div} (\Phi_\e^{(3)} ) \|_{L^2(\Omega_\e^\ell)}
  \le C \sqrt{\e} \| f\|_{C^{1, 1/2}(\Omega)}.
  \end{equation}
  More specifically, 
for each $1\le \ell\le L$,
  the matrix-valued  function $I_\e^\ell$ is a solution of the Stokes equations in $\Omega_\e^\ell$ with $I_\e^\ell =0$ on $\partial \Omega_\e^\ell \setminus \partial \Omega^\ell$.
  On each connected component $\Sigma^k$ of the interface $\Sigma$,  the boundary value of
  $I_\e^\ell$ is either $0$ or given by
  \begin{equation}\label{p4a}
  W_j^- (x/\e) -W_j^+ (x/\e)
  -W_i^- (x/\e)
  (K_{mj}^- -K_{mj}^+ ) \frac{n_i n_m}{ \langle n K^-, n \rangle},
  \end{equation}
  where the repeated indices $i$ and $m$ are summed from $1$ to $d$.
  Here   the subdomains $\Omega^\pm$ are separated by $\Sigma^k$, and $(W^\pm, K^\pm)$ denote the corresponding
  1-periodic matrices for $\Omega^\pm$ and their averages over $Y$, respectively.
  To show $V+\Phi_\e^{(3)}$ is continuous across $\Sigma$, we use the fact that $(\nabla_{\tan} P_0)^+=
  (\nabla_{\tan} P_0)^-$ and
  \begin{equation}\label{p5}
  n \cdot K^+ (f-\nabla P_0)^+=n \cdot K^- (f-\nabla P_0)^-,
  \end{equation}
  where $(v)^\pm  $ denote the trace of $v$ taken from $\Omega^\pm$, respectively.
  The proof of the estimate \eqref{p4} again relies on the observation  that the normal component of \eqref{p4a}
  is of form \eqref{tan-1}.

Theorem \ref{main-thm-2} is proved under the assumption that $\{Y_s^\ell: 1\le \ell \le L\}$ are subdomains of $Y$ with Lipschitz 
boundaries. The $C^{2, 1/2}$ condition and the geometric condition \eqref{condition-1} for $\Omega$ and subdomains $\{ \Omega^\ell\}$
are dictated by the smoothness requirement in its proof for $P_0$  in  each subdomain. 
Note that $P_0$ is a solution of an elliptic equation with piecewise constant coefficients in $\Omega$.
Not much is known about the boundary regularity of $P_0$ if the interface intersects with the boundary $\partial\Omega$ or with each other.

The paper is organized as follows.
In Section \ref{section-2} we collect several useful estimates that are more or  less known.
In Section \ref{section-E} we establish the energy estimates for the Dirichlet problem \eqref{Stokes-1}.
Theorem \ref{main-thm-1} is proved in Section \ref{section-Q}.
In Section \ref{section-C} we give the proof of Theorem \ref{main-thm-2}, assuming the existence of suitable correctors.
Finally, we construct correctors $\Phi_\e^{(1)}$, $\Phi_\e^{(2)}$, and $\Phi_\e^{(3)}$, described above,
  in the last three sections of the paper.
Throughout  the paper we will use $C$ to denote constants that may depend on $d$, $\mu$,
$\Omega$, $\{\Omega^\ell\}$, and $\{Y_s^\ell\}$.
In fact, since the viscosity constant $\mu$  is irrelevant in our study, we will assume $\mu=1$ in the rest of the paper.

  \medskip
  
  \noindent{\bf Acknowledgement.}
  The author is indebt to Professor Xiaoming Wang for raising  the question that is addressed in this paper and
  for several stimulating discussions.
  

\section{Preliminaries}\label{section-2}

Let $Y=[0, 1]^d$ and $\{Y_s^\ell: 1\le \ell \le L \}$ be a finite number of open subsets  of $Y$ with Lipschitz boundaries.
We assume that dist$(\partial Y, \partial Y^\ell_s)>0$ and that $Y_f^\ell=Y \setminus \overline{Y_s^\ell}$ is connected.
Let
$$
\omega^\ell =\bigcup_{z\in \mathbb{Z}^d} (Y_f^\ell +z)
$$
be the periodic repetition of $Y_f^\ell$.
For $1\le j \le d$ and $1\le \ell \le L$, let 
$$
\left(W_j^\ell (y), \pi_j^\ell (y)\right)
=\left(W_{1j}^\ell (y), \dots, W_{dj}^\ell  (y), \pi_j^\ell  (y)\right)\in H^1_{\loc} (\omega^\ell; \R^d) \times L^2_{\loc}(\omega^\ell)
$$
 be the 1-periodic solution of
\begin{equation}\label{cor-1}
\left\{
\aligned
-\Delta W_j^\ell + \nabla \pi_j^\ell  & = e_j & \quad & \text{ in } \omega^\ell,\\
\text{\rm div} (W_j^\ell) & =0& \quad & \text{ in } \omega^\ell,\\
W_j^\ell & =0 & \quad & \text{ on } \partial \omega^\ell,
\endaligned
\right.
\end{equation}
with $\int_{Y^\ell_f} \pi_j^\ell 
\, dy =0$,
where $e_j =(0, \dots, 1, \dots, 0)$ with $1$ in the $j^{th}$ place.
We extend the matrix  $W^\ell=(W_j^\ell )$ to $\R^d$ by zero and define
\begin{equation}\label{K-1}
K^\ell_{ij} =\int_{Y} W_{ij}^\ell  (y)\, dy.
\end{equation}
Since
$$
K_{ij}^\ell =\int_{Y}
\nabla W_{ik}^\ell \cdot \nabla W^\ell_{jk}\, dy
$$
(the repeated  index $k$ is summed from $1$ to $d$),
it follows that the $d\times d$ matrix $K^\ell=(K^\ell_{ij})$ is symmetric and positive definite.

The existence and uniqueness of solutions to \eqref{cor-1} can be proved by applying  the Lax-Milgram Theorem 
on  the closure of the set,
$$
\left\{ u\in C^\infty (\R^d; \R^d):  u \text{ is 1-periodic, } u=0 \text{ in } Y_s^\ell,  \text{ and } \text{\rm div}(u)=0 \text{ in } \R^d\right \},
$$
 in $H^1(Y; \R^d)$.
 By  energy estimates,
\begin{equation}\label{cell-0}
\int_{Y} \left( |\nabla W^\ell|^2 + |W^\ell|^2+ |\pi^\ell|^2  \right) dy \le C,
\end{equation}
where we have also extended $\pi^\ell$ to $\R^d$ by zero.
By periodicity this implies that
\begin{equation}\label{cell-1}
\int_D \left( |\nabla W^\ell (x/\e)|^2 + |W^\ell (x/\e)|^2  + |\pi^\ell (x/\e)|^2 \right) dx \le C,
\end{equation}
where $D$ is a bounded domain and  $C$ depends on diam($D$).

\begin{lemma}\label{lemma-cell}
Let $D$ is a bounded Lipschitz domain in $\R^d$. Then 
\begin{equation}\label{cell-2}
\int_{\partial D }
\left( |\nabla W^\ell (x/\e)|^2 + |W^\ell (x/\e)|^2  + |\pi^\ell (x/\e)|^2 \right) d\sigma \le C,
\end{equation}
where $C$ depends on $D$.
\end{lemma}

\begin{proof}

If $Y_s^\ell$  is of $C^{1, \alpha}$, the inequality above follows directly  from the fact that
$\nabla W^\ell$ and $\pi^\ell$ are bounded in $Y$.
To treat the case where $\partial Y_s^\ell$ is merely  Lipschitz,  
by  periodicity,
 we may assume that $\e=1$ and $D$ is a subdomain of $Y$.
 Note that the bound for the integral of $|W^\ell|^2$ on $\partial D$ follows from \eqref{cell-0}.
 Indeed, if $D$ is a subdomain of $Y$ with Lipschitz boundary,
 $$
 \int_{\partial D}
 |W^\ell|^2\, d\sigma 
 \le C \int_{D} \left( |\nabla W^\ell|^2 + |W^\ell|^2 \right)dy.
 $$
 
 The estimates for $\nabla W^\ell$ and $\pi^\ell$ are a bit more involved.
 By using the fundamental solutions for the Stokes equations in $\R^d$, we may reduce the problem to the estimate
 $$
 \| \nabla u \|_{L^2(\partial D)} + \| p \|_{L^2(\partial D)}
 \le C \left\{
 \| \nabla u \|_{L^2(\widetilde{Y}\setminus Y_s^\ell)} 
 + \| p \|_{L^2(\widetilde{Y}\setminus Y_s^\ell)} 
 + \| h \|_{H^1(\partial Y_s^\ell)}
 \right\},
 $$
 for solutions of the Stokes equations,
 $$
 \left\{
 \aligned
 -\Delta u+\nabla p  & =0   &\quad & \text{ in } \widetilde{Y} \setminus \overline{Y_s^\ell},\\
 \text{\rm div}(u) & =0 & \quad & \text{ in } \widetilde{Y}\setminus \overline{Y_s^\ell},\\
 u &= h & \quad & \text{ on } \partial Y_s^\ell,
 \endaligned
 \right.
 $$
where $h \in H^1(\partial Y_s^\ell; \R^d)$ and $\widetilde{Y}=(1+c)Y$.
The desired estimates follow from the interior estimates as well as  the nontangential-maximal-function estimate,
\begin{equation}\label{non-0}
\| (\nabla u)^*\|_{L^2(\partial Y_s^\ell)}
+ \|(p)^* \|_{L^2(\partial Y_s^\ell)}
\le C \left\{ \| h \|_{H^1(\partial Y_s^\ell)} + \| u\|_{L^2(\widetilde{Y}\setminus Y_s^\ell)}
+ \| p\|_{L^2(\widetilde{Y}\setminus Y_s^\ell)}\right\},
\end{equation}
where the nontangential maximal function $(v)^*$ is defined by
$$
(v)^*(x)
=\sup\left\{ |v(y)|: \ y\in Y \setminus Y_s^\ell \text{ and } |y-x|< C_0\,  \text{\rm dist} (y, \partial Y_s^\ell)\right\}
$$
for $x\in \partial Y_s^\ell$.
The estimate \eqref{non-0} is a consequence  of  the nontangential-maximal-function estimates, established in \cite{FKV}, 
 for solutions of the Dirichlet problem for the Stokes equations in a bounded Lipschitz domain.
\end{proof}

\begin{lemma}\label{skew-lemma}
Fix $1\le j \le d$ and $1\le \ell \le L$.
There exist 1-periodic functions  $\phi_{kij}^\ell (y)$, $i, m=1, 2, \dots, d$, such that $\phi_{kij}^\ell \in H^1(Y)$, $\int_Y \phi_{kij}^\ell\, dy=0$,
\begin{equation}\label{skew-1}
\frac{\partial}{\partial y_k} \left( \phi_{kij}^\ell \right) = W^\ell_{ij} -K^\ell_{ij}
\quad \text{ and } \quad
\phi_{kij}^\ell = -\phi_{ikj}^\ell,
\end{equation}
where the repeated index $k$ is summed from $1$ to $d$.
Moreover,
\begin{equation}\label{skew-2}
\int_{\partial D} |\phi^\ell_{kij} (x/\e)|^2\, d\sigma \le C,
\end{equation}
where $D$ is a bounded Lipschitz domain in $\R^d$ and $C$ depends on $D$.
\end{lemma}

\begin{proof}
See \cite[Lemma 5.3]{Shen-Darcy-1} for the proof of \eqref{skew-1}.
The estimate \eqref{skew-2} follows from the observation, 
$$
\aligned
\|\nabla \phi_{kij}^\ell\|_{L^2(Y)} + \|\phi_{kij}^\ell \|_{L^2(Y)}  & \le  C \| \nabla^2 h_{ij}^\ell \|_{L^2(Y)} + C \|\nabla^2 h_{kj}^\ell\|_{L^2(Y)}\\
& \le C \| W_j^\ell\|_{L^2(Y)} \le C.
\endaligned
$$
\end{proof}

Let $\Omega$ be a bounded Lipschitz domain in $\R^d$ and $\{\Omega^\ell: 1\le \ell\le L\}$ be disjoint subdomains of $\Omega$,
each with Lipschitz boundary,  and satisfying the condition, 
\begin{equation}\label{domain-2}
\overline{\Omega} =\cup_{\ell=1}^L \overline{\Omega^\ell}.
\end{equation}
Define
\begin{equation}\label{K-2}
K =\sum_{\ell=1}^L K^\ell \chi_{\Omega^\ell},
\end{equation}
where  $K^\ell$ is given by \eqref{K-1} and $\chi_{\Omega^\ell}$ denotes the characteristic function of $\Omega^\ell$.

\begin{lemma}\label{lemma-2.1}
Let $f\in L^2(\Omega; \R^d)$.
Then there exists $P_0 \in H^1(\Omega)$, unique up to constants, such that
\begin{equation}\label{2.1-0}
\left\{
\aligned
\text{\rm div} \left( K (f -\nabla P_0)\right) & =0 & \quad & \text{ in } \Omega,\\
n \cdot K (f-\nabla P_0) & =0 & \quad & \text{ on } \partial \Omega,
\endaligned
\right.
\end{equation}
in the sense that
\begin{equation}\label{2.1-1}
\int_\Omega
K(f-\nabla P_0) \cdot \nabla \varphi \, dx
=0
\end{equation}
for any $\varphi \in H^1(\Omega)$.
\end{lemma}

\begin{proof}
This is standard since the coefficient matrix $K$ is positive definite in each subdomain $\Omega^\ell$
and thus in $\Omega$.
\end{proof}

For each $1\le \ell\le L$ and $0< \e<1$, let $\Omega^\ell_\e$ be the perforated domain defined by \eqref{subdomain}, using $Y_s^\ell$.
Let $\Omega_\e$ be given by \eqref{domain-e}.
Note that 
\begin{equation}
\partial \Omega_\e =\partial \Omega \cup \Gamma_\e,
\end{equation}
where $\Gamma_\e=\cup_{\ell=1}^L {\Gamma_\e^\ell} $ 
and $\Gamma_\e^\ell$ consists of the boundaries of holes $\e (Y_s^\ell +z)$  that are removed from $\Omega^\ell$.

\begin{lemma}\label{lemma-2.2}
Let $u \in H^1(\Omega_\e)$ with $u=0$ on $\Gamma_\e$.
Assume $\Gamma^\ell_\e \neq \emptyset$ for all $1\le \ell \le L$. Then
\begin{equation}\label{2.2-0}
\| u \|_{L^2(\Omega_\e)}
\le C \e \| \nabla u_\e \|_{L^2(\Omega_\e)}.
\end{equation}
\end{lemma}

\begin{proof}
It follows from Lemma 2.2 in \cite{Shen-Darcy-1} that for $1 \le \ell \le L$, 
$$
\| u\|^2_{L^2(\Omega_\e^\ell)} \le C  \e^2 \| \nabla u \|^2_{L^2(\Omega_\e^\ell)},
$$
which yields \eqref{2.2-0} by summation.
Note that we do not assume $u=0$ on $\partial \Omega^\ell$.
\end{proof}

From  now on we will assume that $\e>0$ is sufficiently small so that $\Gamma_\e^\ell \neq \emptyset$
for all $1\le  \ell \le L$. The main results in this paper are only relevant for small $\e$.

\begin{lemma}\label{lemma-2.3}
Let $\Omega$ be a bounded Lipschitz domain and $\Omega_\e$ be given by \eqref{domain-e}.
There exists a bounded linear operator,
\begin{equation}\label{R-1}
R_\e: H^1(\Omega; \R^d) \to H^1(\Omega_\e; \R^d), 
\end{equation}
such that
\begin{equation}\label{R-2}
\left\{
\aligned
& R_\e (u)=0 \quad \text{ on } \Gamma_\e \quad \text{ and } \quad R_\e (u)= u \quad \text{ on } \partial \Omega,\\
& R_\e (u) \in H_0^1(\Omega_\e; \R^d) \quad \text{ if } \  u\in H_0^1(\Omega; \R^d),\\
& R_\e (u)=u \quad \text{ in } \Omega \quad \text{ if }  \ u=0 \quad \text{ on } \Gamma_\e,\\
& \text{\rm div} (R_\e (u))=\text{\rm div}(u)   \quad  \text{ in } \Omega_\e \   \text{ if \ \ } \text{\rm div} (u)=0 \quad \text{ in } \Omega\setminus \Omega_\e,
\endaligned
\right.
\end{equation}
and
\begin{equation}\label{R-3}
\e \| \nabla R_\e (u)\|_{L^2(\Omega_\e)}
+ \| R_\e (u)\|_{L^2(\Omega_\e)}
\le C \left\{
\e \| \nabla u \|_{L^2(\Omega)}
+ \| u \|_{L^2(\Omega)} \right\}.
\end{equation}
Moreover,
\begin{equation}\label{R-4}
\| \text{\rm div} (R_\e (u))\|_{L^2(\Omega_\e)}
\le C \| \text{\rm div}(u)\|_{L^2(\Omega)}.
\end{equation}
\end{lemma}

\begin{proof}
A proof for the case $L=1$, which is similar to that of a lemma due to Tartar (in an appendix of \cite{Sanchez-1980}),  
 may be found in \cite[Lemma 2.3]{Shen-Darcy-1}.
 Also see \cite{LA-1990, Allaire-89}.
The same proof works equally well for the case $L\ge 2$.
Indeed, let $u\in H^1(\Omega; \R^d)$.
For each $\e (Y+z)\subset \Omega^\ell$ with $1\le \ell\le L$ and $z\in \mathbb{Z}^d$, we define
$R_\e (u) $ on $\e (Y_f^\ell +z) $ by the Dirichlet problem,
\begin{equation}\label{R-5}
\left\{
\aligned
-\e^2 \Delta R_\e (u) +\nabla q &= -\e^2 \Delta u & \quad & \text{ in } \e (Y_f^\ell +z),\\
\text{\rm div}(R_\e (u)) &= \text{\rm div} (u) +\frac{1}{|\e (Y_f^\ell +z)|}
\int_{\e (Y_s^\ell +z)} \text{\rm div} (u)\, dx & \quad & \text{ in } \e (Y_f^\ell +z),\\
R_\e (u) & =0 & \quad & \text{ on } \partial (\e (Y_s^\ell+z)),\\
R_\e (u) & = u & \quad & \text{ on } \partial (\e (Y+z)).
\endaligned
\right.
\end{equation}
If $x\in \Omega_\e$ and $x\notin \e (Y_f +z)$
for any $\e (Y+z) \subset \Omega^\ell$, we let $R_\e (u)=u$. 
\end{proof}

\begin{lemma}\label{lemma-div}
Let $f\in L^2(\Omega_\e)$ with $\int_{\Omega_\e} f \, dx =0$.
Then there exists $u_\e \in H_0^1(\Omega_\e; \R^d)$ such that 
$ \text{\rm div} (u_\e)=f$ in $\Omega_\e$ and
\begin{equation}\label{div-1a}
\| u_\e \|_{L^2(\Omega_\e)} +
\e \| \nabla u_\e \|_{L^2(\Omega_\e)}
\le C \| f \|_{L^2(\Omega_\e)}.
\end{equation}
\end{lemma}

\begin{proof}
Let $F$ be the zero extension of $f$ to $\Omega$.
Since $F\in L^2(\Omega)$ and $\int_\Omega F\, dx =0$,
there exists $u\in H_0^1(\Omega; \R^d)$ such that 
div$(u)=F$ in $\Omega$ and
$\| u \|_{L^2(\Omega)} + \|\nabla u \|_{L^2(\Omega)} \le C \| F \|_{L^2(\Omega)}$.
Let $u_\e=R_\e (u)$.
Then $u_\e\in H_0^1(\Omega_\e, \R^d)$,  and by \eqref{R-3}, 
$$
\aligned
\e \|\nabla u_\e \|_{L^2(\Omega_\e)}
+\| u_\e \|_{L^2(\Omega_\e)}
 & \le  C \left\{ \e \|\nabla u\|_{L^2(\Omega)} + \| u \|_{L^2(\Omega)} \right\}\\
 & \le C \|f \|_{L^2(\Omega_\e)}.
 \endaligned
 $$
Since div$(u)=F=0$ in $\Omega \setminus \Omega_\e$, by the last  line in \eqref{R-2}, we obtain 
div$(u_\e)=\text{\rm div} (u)=f$ in $\Omega_\e$.
\end{proof}

For $p\in L^2(\Omega_\e)$, as in \cite{LA-1990}, we define an extension  $P$ of $p$ to $L^2(\Omega)$ by 
\begin{equation}\label{ext-1}
P (x)
=\left\{
\aligned
& p(x) & \quad &\text{ if } x\in \Omega_\e, \\
& \fint_{\e (Y_f^\ell+z) } p & \quad  & \text{ if } x\in \e (Y_s^\ell +z) \subset \e (Y+z)\subset \Omega^\ell \text{ for some } 1\le \ell\le L \text{ and } z\in \mathbb{Z}^d.
\endaligned
\right.
\end{equation}

\begin{lemma}\label{lemma-dual}
Let $p\in L^2(\Omega_\e)$ and $P$ be its extension given by \eqref{ext-1}.
Then
\begin{equation}\label{ext-2}
\langle \nabla p, R_\e (u) \rangle_{H^{-1}(\Omega_\e) \times H_0^1(\Omega_\e)}
=\langle \nabla P, u \rangle_{H^{-1}(\Omega) \times H_0^1(\Omega)},
\end{equation}
where $u\in H_0^1(\Omega; \R^d)$ and $R_\e (u)$ is given by Lemma \ref{lemma-2.3}.
\end{lemma}

\begin{proof}
We use an argument found  in \cite{LA-1990, Allaire-89, Allaire-1991}.
Note that if $u \in H_0^1(\Omega; \R^d)$, we have $R_\e (u) \in H_0^1(\Omega_\e; \R^d)$ and 
$$
\aligned
 |\langle \nabla p, R_\e (u) \rangle_{H^{-1} (\Omega_\e) \times H^1_0(\Omega_\e)} |
& = |\langle p, \text{\rm div} (R_\e (u) ) \rangle _{L^2(\Omega_\e)\times L^2(\Omega_\e)}  |\\
& \le \| p \|_{L^2(\Omega_\e)} \| \text{\rm div} (R_\e (u)) \|_{L^2(\Omega_\e)}\\
& \le C \| p \|_{L^2(\Omega_\e)}
\| \text{\rm div}(u)  \|_{L^2(\Omega)},
\endaligned
$$
where we have used the estimate \eqref{R-4} for the last inequality.
Thus  there exists  $\Lambda \in H^{-1} (\Omega; \R^d)$ such that 
$$
\langle \nabla p, R_\e (u) \rangle_{H^{-1}(\Omega_\e) \times H_0^1(\Omega_\e)}
=\langle \Lambda, u \rangle_{H^{-1}(\Omega) \times H_0^1(\Omega)}
$$
for any $u\in H_0^1(\Omega; \R^d)$.
Since $\langle \Lambda, u\rangle =0$ if div$(u)=0$ in $\Omega$,
it follows that $\Lambda =\nabla Q$ for some $Q\in L^2(\Omega)$.

Next, using the fact that $R_\e (u)=u$ for $u\in H_0^1(\Omega_\e; \R^d)$,
we obtain 
$$
\langle \nabla p-\nabla Q, u \rangle_{H^{-1}(\Omega_\e) \times H_0^1(\Omega_\e)} =0
$$
for any $u \in H_0^1(\Omega_\e; \R^d)$.
This implies that $p-Q$ is constant in $\Omega_\e$.
Since $Q$ is only determined up to a constant, we may assume that $Q=p$ in $\Omega_\e$.
Moreover, we note that if $\e (Y +z)\subset \Omega^\ell$  for some $1\le \ell \le L$ and $z\in \mathbb{Z}^d$, 
and $u \in C_0^1(\e (Y_s^\ell+z), \R^d)$, then  $R_\e (u)=0$ in $\Omega_\e$.
It follows  that $\nabla Q=0$  in $\e (Y_s^\ell +z)$.
Thus $Q$ is constant in each $\e (Y_s^\ell+z)$.

Finally,   for any $u \in C_0^1(\e (Y+z); \R^d)$ with $\e (Y+z) \subset \Omega^\ell$, 
we have 
$$
R_\e (u) \in H_0^1 (\e (Y_f^\ell+z); \R^d),
$$
 and 
by \eqref{R-5},
$$
\text{\rm div} (R_\e (u))
=\text{\rm div} (u) +\frac{1}{|\e (Y_f^\ell +z)| }
\int_{\e (Y_s^\ell +z)} \text{\rm div} (u) \, dx
$$
in $\e (Y_f^\ell+z)$. 
This, together with 
$$
\int_{\e (Y_f^\ell +z)} p \cdot \text{\rm div} (R_\e (u))\, dx
=\int_{\e (Y+z)} Q \cdot \text{\rm div} (u)\, dx
$$
and the fact that $Q=p$ in $\Omega_\e$, yields 
$$
\int_{\e (Y_s^\ell +z)} 
\Big( Q-\fint_{\e (Y_f^\ell +z)} p \Big) \text{\rm div} (u)\, dx =0.
$$
Consequently,
$$
Q=\fint_{\e (Y_f^\ell +z)} p \qquad \text{ in } \e (Y_s^\ell +z).
$$
As a result, we have proved that $Q=P$, an extension of $p$ given by \eqref{ext-1}.
\end{proof}


\section{Energy estimates}\label{section-E}

Let $\Omega_\e$ be given by \eqref{domain-e}.
Recall that $\partial \Omega_\e =\partial \Omega \cup \Gamma_\e$, where $\Gamma_\e$ consists of  the boundaries of the
holes of size $\e$  that are removed from $\Omega$.
In this section we establish the energy estimates for the Dirichlet problem,
\begin{equation}\label{D-3-0}
\left\{
\aligned
-\e^2 \Delta u_\e +\nabla p_\e & =f +\e\,  \text{\rm div} (F) & \quad & \text{ in } \Omega_\e,\\
\text{\rm div} (u_\e) & = g & \quad & \text{ in } \Omega_\e,\\
u_\e & =0 & \quad & \text{ on } \Gamma_\e,\\
u_\e & =h & \quad & \text{ on } \partial \Omega,
\endaligned
\right.
\end{equation}
where $(g, h)$ satisfies the compatibility condition,
\begin{equation}\label{com}
\int_{\Omega_\e}g\, dx  =\int_{\partial \Omega} h \cdot n \, d\sigma.
\end{equation}
Throughout this section we assume that $\Omega$, $\Omega^\ell$ and $Y_s^\ell$ for $1\le \ell \le L$ are domains with Lipschitz boundaries.
We use $L^2_0(\Omega_\e)$ to denote the subspace of  functions in $L^2(\Omega_\e)$ with mean value zero.

\begin{thm}\label{e-thm}
Let $f\in L^2(\Omega_\e; \R^d)$ and $F \in L^2(\Omega_\e; \R^{d\times d})$.
Let $g \in L^2(\Omega_\e)$ and $h\in H^{1/2} (\partial \Omega; \R^d)$ satisfy the condition \eqref{com}.
Let $(u_\e, p_\e)\in H^1(\Omega_\e; \R^d) \times L^2_0(\Omega_\e)$ be a weak solution of \eqref{D-3-0}.
Then
\begin{equation}\label{3.1-0}
\aligned
 & \e \| \nabla u_\e \|_{L^2(\Omega_\e)}
+ \| u_\e \|_{L^2(\Omega_\e)}
+ \| p_\e \|_{L^2(\Omega_\e)}\\
& \le C \left\{
\| f\|_{L^2(\Omega_\e)} + \| F \|_{L^2(\Omega_\e)}
+\| g \|_{L^2(\Omega_\e)}
+ \| H \|_{L^2(\Omega)} +\| \text{\rm div} (H) \|_{L^2(\Omega)}
+ \e \| \nabla H \|_{L^2(\Omega)} \right \},
\endaligned
\end{equation}
where $H$ is any function in $H^1(\Omega; \R^d)$ with the property  $H=h$ on $\partial\Omega$.
\end{thm}

\begin{proof}
This theorem was proved in \cite[Section 3]{Shen-Darcy-1} for the case $L=1$.
The proof for the case $L\ge 2$ is similar.
We provide a proof here for the reader's convenience.

Step 1. We  show that
\begin{equation}\label{3.1-1}
\| p_\e\|_{L^2(\Omega_\e)}
\le C \left\{ \e \|\nabla u_\e \|_{L^2(\Omega_\e)}
+ \| f\|_{L^2(\Omega_\e)} + \| F \|_{L^2(\Omega_\e)} \right\}.
\end{equation}
To this end we use Lemma \ref{lemma-div} to find $v_\e\in H_0^1(\Omega_\e; \R^d)$ such that
$\text{\rm div} (v_\e)=p_\e$ in $\Omega_\e$ and
\begin{equation}\label{3.1-2}
\e \| \nabla v_\e \|_{L^2(\Omega_\e)}
+ \| v_\e \|_{L^2(\Omega_\e)}
\le C \| p_\e \|_{L^2(\Omega_\e)}.
\end{equation}
By using $v_\e$ as a test function we obtain 
$$
\aligned
\| p_\e\|_{L^2(\Omega_\e)}^2
 & \le \e^2 \|\nabla u_\e \|_{L^2(\Omega_\e)}
 \|\nabla v_\e\|_{L^2(\Omega_\e)}
 + \| f\|_{L^2(\Omega_\e)} \| v_\e\|_{L^2(\Omega_\e)}
 + \e \| F \|_{L^2(\Omega_\e)} \| \nabla v_\e \|_{L^2(\Omega_\e)}\\
 & \le C \| p_\e \|_{L^2(\Omega_\e)}
 \left\{ \e \| \nabla u_\e\|_{L^2(\Omega_\e)}
 + \| f\|_{L^2(\Omega_\e)} + \| F \|_{L^2(\Omega_\e)} \right\},
 \endaligned
 $$
 where we have used \eqref{3.1-2} for the last inequality.
This yields \eqref{3.1-1}.

Step 2. We prove \eqref{3.1-0} in the case $h=0$.
In this case we may use $u_\e\in H_0^1(\Omega_\e; \R^d)$ as a test function to obtain
$$
\aligned
\e^2 \|\nabla u_\e\|^2_{L^2(\Omega_\e)}
\le \| p_\e \|_{L^2(\Omega_\e)} \| g\|_{L^2(\Omega_\e)}
+ \| f\|_{L^2(\Omega_\e)} \| u_\e\|_{L^2(\Omega_\e)}
+ \e \| F \|_{L^2(\Omega_\e)} \| \nabla u_\e \|_{L^2(\Omega_\e)}.
\endaligned
$$
By using the Cauchy inequality as well as the estimate 
$\|u_\e\|_{L^2(\Omega_\e)} \le C \e \| \nabla u_\e\|_{L^2(\Omega_\e)}$,
we deduce that
$$
\| u_\e\|_{L^2(\Omega_\e)}
+ \e \|\nabla u_\e\|_{L^2(\Omega_\e)}
\le C \left\{ \| p_\e\|^{1/2} _{L^2(\Omega_\e)}  \| g\|^{1/2}_{L^2(\Omega_\e)}
+ \| f\|_{L^2(\Omega_\e)}
+ \| F \|_{L^2(\Omega_\e)} \right\}.
$$
This, together with \eqref{3.1-1}, gives \eqref{3.1-0} for the case $h=0$.

Step 3. We consider the general case $h \in H^{1/2}(\partial\Omega; \R^d)$.
Let $H$ be a function in $ H^1(\Omega; \R^d)$ such that $H=h$ on $\partial\Omega$.
Let $w_\e = u_\e -R_\e (H)$, where $R_\e(H)$ is given by Lemma \ref{lemma-2.3}.
Then $w_\e\in H_0^1(\Omega_\e; \R^d)$ and 
$$
\left\{
\aligned
-\e^2 \Delta w_\e +\nabla p_\e & = f+\e\,  \text{\rm div}(F) +\e^2 \Delta R_\e (H),\\
\text{\rm div} (w_\e) & =g-\text{\rm div} (R_\e (H)),
\endaligned
\right.
$$
in $\Omega_\e$.
By Step 2 we obtain 
$$
\aligned
&\e \|\nabla w_\e \|_{L^2(\Omega_\e)} + \|w_\e \|_{L^2(\Omega_\e)}
+ \| p_\e \|_{L^2(\Omega_\e)}\\
& \le C 
\left\{
\| f\|_{L^2(\Omega_\e)} + \| F \|_{L^2(\Omega_\e)} + \e \|\nabla R_\e(H)\|_{L^2(\Omega_\e)} 
+ \| g \|_{L^2(\Omega_\e)} + \| \text{\rm div} (R_\e (H))\|_{L^2(\Omega_\e)}
\right\}.
\endaligned
$$
It follows that
$$
\aligned
&\e \|\nabla u_\e \|_{L^2(\Omega_\e)} + \|u_\e \|_{L^2(\Omega_\e)}
+ \| p_\e \|_{L^2(\Omega_\e)}\\
& \le C 
\Big\{
\| f\|_{L^2(\Omega_\e)} + \| F \|_{L^2(\Omega_\e)} +  \| g \|_{L^2(\Omega_\e)} \\
& \qquad\qquad  + \e \|\nabla R_\e(H)\|_{L^2(\Omega_\e)} 
 +\| R_\e (H)\|_{L^2(\Omega_\e)} 
+ \| \text{\rm div} (R_\e (H))\|_{L^2(\Omega_\e)}
\Big\}\\
&\le C \Big\{
\| f\|_{L^2(\Omega_\e)} + \| F \|_{L^2(\Omega_\e)} +  \| g \|_{L^2(\Omega_\e)} 
 + \e \|\nabla H \|_{L^2(\Omega)} 
 +\| H \|_{L^2(\Omega)} 
+ \| \text{\rm div} (H)\|_{L^2(\Omega)}
\Big\},
\endaligned
$$
where we have used estimates \eqref{R-3} and \eqref{R-4} for the last inequality.
\end{proof}

\begin{cor}\label{cor-e}
Let $(u_\e, p_\e)$ be the same as in Theorem \ref{e-thm}.
Then
\begin{equation}\label{3.2-0}
\aligned
&\e \|\nabla u_\e \|_{L^2(\Omega_\e)} + \|u_\e \|_{L^2(\Omega_\e)}
+ \| p_\e \|_{L^2(\Omega_\e)}\\
& \le C 
\Big\{
\| f\|_{L^2(\Omega_\e)} + \| F \|_{L^2(\Omega_\e)} +  \| g \|_{L^2(\Omega_\e)} 
+ \| h \|_{L^2(\partial \Omega)} + \e  \| h \|_{H^{1/2} (\partial \Omega)} \Big\}.
\endaligned
\end{equation}
\end{cor}

\begin{proof}
For $h \in H^{1/2} (\partial \Omega; \R^d)$,
let $H$ be the weak solution in $H^1(\Omega; \R^d)$ of the Dirichlet problem, 
$$
\left\{
\aligned
-\Delta H +\nabla q & =0&\quad & \text{ in } \Omega,\\
\text{\rm div} (H) & =\gamma & \quad & \text{ in } \Omega,\\
u & =h & \quad & \text{ on } \partial \Omega,
\endaligned
\right.
$$
where the constant 
$$
\gamma =\frac{1}{|\Omega|} \int_{\partial\Omega} h \cdot n \, d\sigma 
$$
is chosen so that the compatibility condition \eqref{com} is satisfied.
Note that 
$$
\|\text{\rm div}(H)\|_{L^2(\Omega)}
=C |\gamma |\le C \| h\|_{L^2(\partial\Omega)},
$$
and by the standard energy estimates,
$
\|\nabla H \|_{L^2(\Omega)}  \le C \| h \|_{H^{1/2}(\partial\Omega)}.
$
In view of \eqref{3.1-0} we only need to show that 
\begin{equation}\label{3.2-0a}
 \| H \|_{L^2(\Omega)}
\le C 
\| h\|_{L^2(\partial\Omega)}.
\end{equation}

To this end, let
$$
H_1=H-\gamma (x-x_0)/d,
$$
where $x_0\in \Omega$.
Since $-\Delta H_1 +\nabla q=0$ and div$(H_1)=0$ in $\Omega$,
it follows from \cite{FKV} that
$$
\aligned
\| H_1\|_{L^2(\Omega)}
 &\le C \| (H_1)^* \|_{L^2(\partial\Omega)}\\
 &\le C \| H_1\|_{L^2(\partial \Omega)}
\le C \| h\|_{L^2(\partial\Omega)},
\endaligned
$$
where $(H_1)^*$ denotes the nontangential maximal function of $H_1$.
As a result, we obtain 
$$
\aligned
\| H\|_{L^2(\Omega)}
 & \le \| H_1 \|_{L^2(\Omega)} + C |\gamma| \\
 &\le C \| h\|_{L^2(\partial\Omega)},
\endaligned
$$
which completes the proof.
\end{proof}

\begin{cor}\label{cor-p}
Let $(u_\e, p_\e)$ be the same as in Theorem \ref{e-thm}.
Let $P_\e$ be the extension of $p_\e$, defined by \eqref{ext-1}.
Then
\begin{equation}\label{cor-p-0}
\|P_\e \|_{L^2(\Omega)}
\le C \left\{ 
\| f\|_{L^2(\Omega_\e)}
+ \| F \|_{L^2(\Omega_\e)}
+ \| g\|_{L^2(\Omega_\e)}
+  \| h \|_{L^2(\partial \Omega)}
+ \e \| h\|_{H^{1/2} (\partial \Omega)} \right\}.
\end{equation}
\end{cor}

\begin{proof}
By the definition of $P_\e$, we have
$$
\aligned
\int_\Omega |P_\e|^2\, dx
 & =\int_{\Omega_\e} |p_\e|^2\, dx
+ \sum_{\ell=1}^L
\sum_z |\e (Y_s^\ell +z)| \Big(\fint_{\e (Y^\ell_f +z)} p_\e \Big)^2\\
& \le \sum_{\ell=1}^L \frac{1}{|Y_f^\ell|} \int_{\Omega_\e^\ell} |p_\e|^2\, dx,
\endaligned
$$
which, together with \eqref{3.2-0}, gives \eqref{cor-p-0}.
\end{proof}


\section{Homogenization and proof of Theorem \ref{main-thm-1}}\label{section-Q}

Let $f\in L^2(\Omega; \R^d)$ and $h\in H^{1/2}(\partial\Omega; \R^d)$ with $\int_{\partial\Omega} h\cdot n\, d\sigma =0$, where $n$ denotes the outward unit normal to 
$\partial\Omega$.
Consider the Dirichlet problem,
\begin{equation}\label{D-4}
\left\{
\aligned
-\e^2 \Delta u_\e +\nabla p_\e & =f & \quad & \text{ in } \Omega_\e,\\
\text{\rm div} (u_\e) & =0 & \quad & \text{ in } \Omega_\e,\\
u_\e & =0 & \quad & \text{ on } \Gamma_\e,\\
u_\e  & =h & \quad & \text{ on } \partial \Omega,
\endaligned
\right.
\end{equation}
where $\Omega_\e$ is  given by \eqref{domain-e} and $\partial \Omega_\e=\partial \Omega \cup \Gamma_\e$.
Throughout the section we assume that $\Omega$, $\Omega^\ell$ and $Y_s^\ell$ for $1\le \ell \le L$,
are domains with Lipschitz boundaries.
As before, we extend $u_\e$ to $\Omega$ by zero and still denote the extension by $u_\e$.
 We use  $P_\e$ to denote the extension of $p_\e$ to $\Omega$, given by \eqref{ext-1}.
The goal of this section is to prove the following theorem, which contains Theorem \ref{main-thm-1}
as a special case $h=0$.

\begin{thm}\label{H-thm}
Let $f\in L^2(\Omega; \R^d)$ and $h \in H^{1/2}(\partial\Omega; \R^d)$ with
 $\int_{\partial \Omega} h \cdot n \, d\sigma =0$.
Let $(u_\e, p_\e)\in H^1(\Omega_\e; \R^d) \times L_0^2(\Omega_\e)$ be the weak solution of \eqref{D-4}.
Let  $(u_\e, P_\e)$ be the extension of $(u_\e, p_\e)$.
Then $u_\e \to u_0$ weakly in $L^2(\Omega; \R^d)$ and
$P_\e -\fint_\Omega P_\e \to P_0$ strongly in $L^2(\Omega)$, as $\e\to 0$, where
$P_0  \in H^1(\Omega)$, $ \int_\Omega P_0\, dx =0$,
 $(u_0, P_0)$ 
is governed  by a Darcy law,
\begin{equation}\label{Darcy-4}
\left\{
\aligned
u_0 & = K (f-\nabla P_0) &\quad & \text{ in } \Omega,\\
\text{\rm div} (u_0) & =0& \quad & \text{ in } \Omega,\\
u_0\cdot n &  = h\cdot n & \quad & \text{ on } \partial \Omega,
\endaligned
\right.
\end{equation}
with the permeability  matrix $K$  given by \eqref{K-0}.
\end{thm}

We begin with the strong convergence of $P_\e$.

\begin{lemma}\label{lemma-4.1}
Let $(u_{\e_k}, p_{\e_k})$ be a weak solution of \eqref{D-4} with $\e=\e_k$.
Suppose that as $\e_k \to 0$,
$P_{\e_k} \to P$ weakly in $L^2(\Omega)$ for some $P\in L^2(\Omega)$.
Then $P_{\e_k}  \to P$ strongly in $L^2(\Omega)$.
\end{lemma}

\begin{proof}
The proof is similar to that for the classical case $L=1$ (see e.g. \cite{Allaire-1997}).
One argues by contradiction.
Suppose that $P_{\e_k}$ does not converge strongly to $P$ in $L^2(\Omega)$.
Since
$$
\|\nabla P_{\e_k} -\nabla P\|_{H^{-1}(\Omega)}
\sim \| P_{\e_k} -P-\fint_\Omega (P_{\e_k} - P) \|_{L^2(\Omega)}
$$
and $\int_\Omega P_{\e_k} \, dx \to \int_\Omega P\, dx$,
it follows that $\nabla P_{\e_k}$ does not converge to $\nabla P$ strongly in $H^{-1}(\Omega; \R^d)$.
By passing to a subsequence, this implies that there exists a sequence $\{\psi_k \} \subset H_0^1(\Omega; \R^d)$ such that
$\| \psi_k \| _{H^1_0(\Omega)} =1$ and
$$
|\langle \nabla P_{\e_k} -\nabla P, \psi_k \rangle_{H^{-1}(\Omega) \times H^1_0(\Omega)}|
\ge c_0>0.
$$
By passing to another  subsequence, we may assume that $\psi_k \to \psi_0$ weakly in $H_0^1(\Omega; \R^d)$.
Let $\varphi_k =\psi_k -\psi_0$.
Using  $P_{\e_k} \to P$ weakly in $L^2(\Omega)$, we obtain 
\begin{equation}\label{4.1-1}
|\langle \nabla P_{\e_k} -\nabla P, \varphi_k \rangle_{H^{-1}(\Omega) \times H^1_0(\Omega)}|
\ge c_0/2,
\end{equation}
if $k$ is sufficiently large. 
Since $\varphi_k \to 0$ weakly in $H^1_0(\Omega; \R^d)$, we may conclude further  that
\begin{equation}\label{4.1-2}
|\langle \nabla P_{\e_k} , \varphi_k \rangle_{H^{-1}(\Omega) \times H^1_0(\Omega)}|
\ge c_0/4,
\end{equation}
if $k$ is sufficiently large.
On the other hand, by \eqref{lemma-dual}, we have
\begin{equation}\label{4.1-3}
\aligned
& |\langle \nabla P_{\e_k}, \varphi_k \rangle_{H^{-1} (\Omega) \times H_0^1(\Omega)} |
 =| \langle \nabla p_{\e_k}, R_{\e_k} (\varphi_k) \rangle_{H^{-1}(\Omega_{\e_k}) \times H_0^1(\Omega_{\e_k})} |\\
& =| \langle \e_k^2 \Delta u_{\e_k} + f , R_{\e_k} (\varphi_k) \rangle_{H^{-1}(\Omega_{\e_k} ) \times H_0^1(\Omega_{\e_k} )} |\\
&\le \e_k^2 \|\nabla u_{\e_k} \|_{L^2(\Omega_{\e_k})}
\|\nabla R_{\e_k} (\varphi_k)\|_{L^2(\Omega_{\e_k})}
+ \| f\|_{L^2(\Omega)} \| R_{\e_k} (\varphi_k)\|_{L^2(\Omega_{\e_k})}\\
&\le C \left( \| f\|_{L^2(\Omega)}
+ \| h \|_{H^{1/2}(\partial\Omega)} \right)
\left( \e_k \|\nabla R_{\e_k} (\varphi_k)\|_{L^2(\Omega_{\e_k})}
+ \| R_{\e_k} (\varphi_k)\|_{L^2(\Omega_{\e_k})} \right)\\
&\le 
C \left( \| f\|_{L^2(\Omega)}
+ \| h \|_{H^{1/2}(\partial\Omega)} \right)
\left( \e_k \|\nabla \varphi_k\|_{L^2(\Omega)}
+ \| \varphi_k\|_{L^2(\Omega)} \right),
\endaligned
\end{equation}
where we have used the estimate \eqref{3.2-0} for the second inequality and \eqref{R-3} for the last.
This contradicts with \eqref{4.1-2} as the right-hand side of \eqref{4.1-3} goes to zero.
\end{proof}

By Corollaries  \ref{cor-e} and \ref{cor-p}, the sets $\{ u_\e: 0< \e< 1 \}$ and $\{ P_\e: 0< \e< 1 \}$ are bounded 
in $L^2(\Omega; \R^d)$ and $L^2 (\Omega)$, respectively.
It follows that for any sequence $\e_k \to 0 $, there exists a subsequence, still denoted by $\{\e_k\}$,
 such that $u_{\e_k} \to u$ and $P_{\e_k} \to P$
weakly in $L^2(\Omega; \R^d)$ and $L^2(\Omega)$, respectively.
By Lemma \ref{lemma-4.1},  $P_{\e_k} \to P$ strongly in $L^2(\Omega)$. Thus,
as in the classical case $L=1$,
to prove Theorem \ref{H-thm}, it suffices to show that if 
$\e_k \to 0$, $u_{\e_k} \to u$ weakly in $L^2(\Omega; \R^d)$, and $P_{\e_k} \to P$ strongly  in
$L^2(\Omega)$, then  $P\in H^1(\Omega)$ and $(u, P)$ is a weak solution of 
\eqref{Darcy-4}.
Since the solution of \eqref{Darcy-4} is unique under the conditions that $P_0\in H^1(\Omega)$ and $ \int_\Omega P_0\, dx =0$, one concludes that
as $\e \to 0$, $u_\e\to u_0$ weakly in $L^2(\Omega; \R^d)$ and
$P_\e -\fint_\Omega P_\e \to P_0$ strongly in $L^2(\Omega)$, where $(u_0, P_0)$ is the 
unique solution of \eqref{Darcy-4} with the property $P_0\in H^1(\Omega)$ and $\int_\Omega P_0\, dx =0$.

\begin{lemma}\label{lemma-4.2}
Let $\{ \e_k \}$ be a sequence such that $\e_k \to 0$.
Suppose that $u_{\e_k} \to u$ weakly in $L^2(\Omega; \R^d)$ and
$P_{\e_k} \to P$ strongly in $L^2(\Omega)$.
Then $P\in H^1(\Omega^\ell)$ for $1\le \ell \le L$ and  $(u, P)$ is a solution of $\eqref{Darcy-4}$.
\end{lemma}

\begin{proof}

Since
$$
\int_{\Omega} u_{\e_k} \cdot \nabla \varphi\, dx
=\int_{\partial \Omega} (h\cdot n) \varphi\, d\sigma
$$
for any $\varphi \in C^\infty(\R^d)$, by letting $k \to \infty$, we see that 
$$
\int_{\Omega} u \cdot \nabla \varphi\, dx
=\int_{\partial \Omega} (h\cdot n) \varphi\, d\sigma
$$
for any $\varphi \in C^\infty(\R^d)$.
It follows that div$(u)=0$ in $\Omega$ and $u\cdot n = h\cdot n$ on $\partial \Omega$.

Next, we show that $P\in H^1(\Omega^\ell)$ for each subdomain $\Omega^\ell$ and that
\begin{equation}\label{4.2-0}
u=K^\ell (f-\nabla P) \quad \text{ in } \Omega^\ell,
\end{equation}
 where $K^\ell =(K^\ell_{ij})$ is defined by \eqref{K-1}.
The argument is the same as that of Tartar  for the case $L=1$ (see \cite{Sanchez-1980}).
Fix $1\le \ell \le L$, $1\le j\le d$,  and $\varphi \in C_0^\infty(\Omega^\ell)$.
We assume $k>1$ is sufficiently large that 
supp$(\varphi) \subset \{ x\in \Omega^\ell: \text{dist}(x, \partial\Omega^\ell) \ge C_d \e_k \}$.
Let $( W_j^\ell (y), \pi_j^\ell (y)) $ be the 1-periodic functions given by \eqref{cor-1}.
By using $W_j^\ell (x/\e_k) \varphi$ as a test function, we obtain 
\begin{equation}\label{4.2-1}
\aligned
& \e_k\int_{\Omega^\ell}
\nabla u_{\e_k}  \cdot \nabla W_j^\ell (x/\e_k ) \varphi\, dx
+ \e^2_k \int_{\Omega^\ell}
\nabla u_{\e_k} \cdot W_j^\ell (x/\e_k) \nabla \varphi \, dx
-\int_{\Omega^\ell} P_{\e_k} W_j^\ell (x/\e_k) \cdot \nabla \varphi\, dx\\
&=\int_{\Omega^\ell}
f\cdot W_j^\ell (x/\e_k) \varphi\, dx,
\endaligned
\end{equation}
where we have used the facts that $\text{\rm div} (W^\ell_j (x/\e))=0$ in $\R^d$ and
$W_j^\ell (x/\e)=0$ on $\Gamma_{\e} $.
Since $W_{ij}^\ell (x/\e_k) \to K^\ell_{ij}$ weakly in $L^2(\Omega^\ell)$ and
$P_{\e_k} \to P$ strongly in $L^2(\Omega^\ell)$,
we deduce from \eqref{4.2-1}  that
\begin{equation}\label{4.2-1-1}
\lim_{k \to \infty}
\e_k\int_{\Omega^\ell}
\nabla u_{\e_k}  \cdot \nabla W_j^\ell (x/\e_k ) \varphi\, dx
=\int_{\Omega^\ell} P K^\ell_{ij} \frac{\partial \varphi}{\partial x_i}\, dx
+\int_{\Omega^\ell} f_iK_{ij}^\ell \varphi\, dx,
\end{equation}
where the repeated  index $i$ is summed from $1$ to $d$.

Note that
$$
\aligned
-\e^2 \Delta \left( W_j^\ell (x/\e) \right)
+\nabla \left( \e \pi_j^\ell (x/\e) \right)  & =e_j  \\
\endaligned
$$
in the set $\{ x\in \Omega^\ell_\e: \text{\rm dist}(x, \partial \Omega^\ell) \ge c_d \e \}$.
By using $u_{\e_k} \varphi$ as a test function, we see that
\begin{equation}\label{4.2-2}
\aligned
& \e_k \int_{\Omega^\ell} \nabla  W_j^\ell (x/\e_k)  \cdot(  \nabla u_{\e_k} ) \varphi\, dx
+\e_k \int_{\Omega^\ell} \nabla  W_j^\ell (x/\e_k) \cdot u_{\e_k} (\nabla \varphi)\, dx\\
& \qquad\qquad-\e_k \int_{\Omega^\ell} \pi_j^\ell (x/\e_k) u_{\e_k} (\nabla \varphi)\, dx
=\int_{\Omega^\ell} e_j \cdot u_{\e_k} \varphi\, dx,
\endaligned
\end{equation}
which  leads to
\begin{equation}\label{4.2-3}
\lim_{k \to \infty}
\e_k \int_{\Omega^\ell} \nabla  W_j^\ell (x/\e_k)  \cdot(  \nabla u_{\e_k} ) \varphi\, dx.
=\int_{\Omega^\ell} e_j \cdot u \varphi\, dx.
\end{equation}
In view of \eqref{4.2-1-1} and \eqref{4.2-3} we obtain
$$
\int_{\Omega^\ell} e_j \cdot u \varphi\, dx
=\int_{\Omega^\ell} P K^\ell_{ij} \frac{\partial \varphi}{\partial x_i}\, dx
+\int_{\Omega^\ell} f_iK_{ij}^\ell \varphi\, dx.
$$
Since $\varphi\in C_0^\infty(\Omega^\ell)$ is arbitrary and the constant matrix  $K^\ell=(K_{ij}^\ell)$ is
invertible, we conclude that $P\in H^1(\Omega^\ell)$
and
$$
u_j = K_{ij}^\ell \Big(f_i -\frac{\partial P}{\partial x_i}\Big)
$$
in $\Omega^\ell$.
Since $K^\ell $ is also symmetric, this gives \eqref{4.2-0}.
\end{proof}

To prove the effective pressure  in Lemma \ref{lemma-4.2}  $P\in H^1(\Omega)$,
it remains  to show that $P$ is  continuous across the interface $\Sigma =\Omega\setminus \cup_{\ell=1}^L \Omega^\ell$
between subdomains.

\begin{lemma}\label{lemma-4.3}
Let $f\in C^m (B(x_0, 2c\e); \R^d)$ for some $x_0\in \R^d$, $m\ge 0$ and $c>0$.
Suppose that 
\begin{equation}\label{4.3-0}
\left\{
\aligned
-\e^{2}  \Delta u_\e +\nabla p_\e & = f & \quad & \text{ in } B(x_0, 2c\e),\\
\text{\rm div} (u_\e) & =0 & \quad & \text{ in } B(x_0, 2c\e).
\endaligned
\right.
\end{equation}
Then
\begin{equation}\label{4.3-1}
\aligned
 & \e^{m+2}  \left(\fint_{B(x_0, c\e)} |\nabla^{m+2} u_\e|^2\right)^{1/2}
 \le C \left(\fint_{B(x_0, 2c\e)} |u_\e|^2 \right)^{1/2}
+ C \sum_{k=0}^m \e^k\| \nabla^k f \|_\infty,
\endaligned
\end{equation}
where $C$ depends only on $d$, $m$ and $c$.
\end{lemma}

\begin{proof}
The case $\e=1$ is given by  interior estimates for the Stokes equations.
The general case follows by a simple rescaling argument.
\end{proof}

Define
\begin{equation}\label{gamma-1}
\gamma_\e = \big\{ x\in \Sigma: \ \text{\rm dist} (x, \partial \Omega)\ge  \e \big\},
\end{equation}
where $\Sigma $ is the interface given by \eqref{interface}

\begin{lemma}\label{Lemma-4.4}
Let $(u_\e, p_\e)$ be a solution of \eqref{D-4} with $f\in C^\infty (\R^d; \R^d)$ and $h \in H^{1/2}(\partial\Omega; \R^d)$.
Then, for $m\ge 0$, 
\begin{equation}\label{4.4-0}
\aligned
\| \nabla^m u_\e \|_{L^2 (\gamma_\e)}   & \le C (f, h)  \e^{-m-\frac12}, \\
\| p_\e\|_{L^2(\gamma_\e)} &  \le C(f, h)  \e^{-\frac12}, \\
 \|\nabla p_\e \|_{L^2(\gamma_\e)}  & \le C(f, h)  \e^{-\frac12},
\endaligned
\end{equation}
where $C(f, h) $ depends on $m$, $f$ and $h$,  but not on $\e$.
\end{lemma}

\begin{proof}
Recall that  
$$
\Sigma = \cup_{\ell=1}^L \partial\Omega^\ell \setminus \partial\Omega.
$$
It follows that  $\gamma_\e =\cup_{\ell=1}^L \gamma_\e^\ell$, where
$$
\gamma_\e^\ell = \big\{ x\in \partial \Omega^\ell: \ \text{dist} (x, \partial\Omega)\ge \e \big\}.
$$
Thus,  it  suffices to prove  \eqref{4.4-0} with $\gamma_\e^\ell $ in the place of $\gamma_\e$.
Let
$$
D^\ell_\e =\left\{ x\in \Omega^\ell: \text{ dist} (x, \gamma_\e^\ell) < c\, \e \right\}.
$$
Using the assumption that $\Omega^\ell$ is a bounded Lipschitz domain, one may show that
\begin{equation}\label{4.4-4}
\aligned
\int_{\gamma_\e^\ell} 
|\nabla ^ m u_\e|^2\, d\sigma
 & \le  \frac{C}{\e} \int_{D_\e^\ell} |\nabla^m u_\e|^2\, dx + C \e \int_{D_\e^\ell} |\nabla^{m+1} u_\e|^2\, dx\\
& \le \frac{C}{\e^{1+2m}}  \left\{ \int_{\Omega_\e} |u_\e|^2\, dx + C(f)  \right\},
\endaligned
\end{equation}
where $C(f)$ depends on $f$.
We point out that the second inequality in \eqref{4.4-4} follows by
 covering $D_\e^\ell$ with balls of radius $ c\e$ and using \eqref{4.3-1}.
This, together with the energy estimate \eqref{3.2-0}, yields 
$$
\|\nabla^m  u_\e \|_{L^2(\gamma_\e^\ell)}
\le C (f, h) \e^{-m -\frac12},
$$
where $C(f, h)$ depends on $f$ and $h$.
Next, using the equation $-\e^2 \Delta u_\e +\nabla p_\e =f$, we obtain 
$$
\aligned
\|\nabla p_\e \|_{L^2(\gamma_\e^\ell)}
 & \le \e^2 \| \Delta u_\e \|_{L^2(\gamma_\e^\ell)}
+ \| f\|_{L^2(\gamma_\e^\ell)}\\
& \le C (f, h) \e^{-1/2}.
\endaligned
$$
Finally,  observe  that
$$
\aligned
\int_{\gamma_\e^\ell}
|p_\e|^2\, d\sigma
 & \le \frac{C}{\e} \int_{D_\e^\ell} |p_\e|^2\, dx
+ C \e \int_{D_\e^\ell} |\nabla p_\e|^2\, dx\\
& \le \frac{C}{\e} \int_{\Omega_\e} | p_\e|^2 \, dx 
+ C \e^ 5 \int_{D_\e^\ell } |\Delta u_\e|^2\, dx + C (f),\\
& \le \frac{C}{\e} \int_{\Omega_\e} |p_\e|^2\, dx
+ C \e \int_{\Omega_\e} |u_\e|^2\, dx + C (f).
\endaligned
$$
This, together with the energy estimate  \eqref{3.2-0}, yields the second inequality in \eqref{4.4-0}.
\end{proof}

The following is the main technical lemma in  the proof of Theorem \ref{H-thm}.

\begin{lemma}\label{lemma-4.5}
Let $(u_{\e_k}, p_{\e_k})$, $P_{\e_k}$,   and $(u, P)$ be the same as in Lemma \ref{lemma-4.2}.
Also assume that $f\in C^\infty(\R^d; \R^d)$.
Let $P^\ell$ denote the trace of $P$, as a function in $H^1 (\Omega^\ell)$, 
on $\partial\Omega^\ell$.
Then, for any $\varphi \in C_0^\infty (\Omega)$,
\begin{equation}\label{4.5-0}
\int_{\partial \Omega^\ell} n_j P^\ell \varphi\, dx
=\lim_{k \to \infty} 
\int_{\partial \Omega^\ell} n_j p_{\e_k} \varphi\, d \sigma,
\end{equation}
where $1\le \ell \le L$, $1\le j\le d$,  and
$n=(n_1, n_2, \dots, n_d)$ denotes the outward unit normal to $\partial\Omega^\ell$.
\end{lemma}

\begin{proof}
For notational simplicity we use $\e$ to denote $\e_k$.
Fix $1\le j\le d$ and $1\le \ell \le L$.
Let $\varphi \in C_0^\infty(\Omega)$.
Then
$$
\aligned
 & \e^2\int_{\Omega_\e^\ell} \nabla u_\e \cdot \nabla \left( W_j^\ell  (x/\e) \varphi \right)\, dx\\
& \qquad
=\e \int_{\Omega_\e^\ell}
\nabla u_\e \cdot \nabla W_j^\ell (x/\e) \varphi\, dx
+\e^2 \int_{\Omega_\e^\ell}  \nabla u_\e \cdot W_j^\ell (x/\e) (\nabla \varphi)\, dx,
\endaligned
$$
and by integration by parts,
$$
\aligned
 & \e^2\int_{\Omega_\e^\ell} \nabla u_\e \cdot \nabla \left( W_j^\ell  (x/\e) \varphi \right)\, dx\\
&=\int_{\Omega^\ell} f \cdot W_j^\ell (x/\e) \varphi\, dx
+\int_{\Omega^\ell} P_\e W_j^\ell (x/\e) \cdot \nabla \varphi\, dx
+\int_{\partial \Omega^\ell}
\frac{\partial u_\e}{\partial \nu} \cdot W_j^\ell (x/\e) \varphi\, d\sigma,
\endaligned
$$
where
$$
\frac{\partial u_\e}{\partial \nu} =\e^2 \frac{\partial u_\e}{\partial n} - p_\e n.
$$
By letting $\e\to 0$ we obtain
\begin{equation}\label{4.5-2}
\aligned
 & \lim_{\e \to 0}
\e \int_{\Omega_\e^\ell}
\nabla u_\e \cdot \nabla W_j^\ell (x/\e) \varphi\, dx\\
& =\int_{\Omega^\ell} f \cdot K_j^\ell \varphi\, dx
+\int_{\Omega^\ell}
P K_j^\ell \cdot \nabla \varphi\, dx
+\lim_{\e\to 0}
\int_{\partial \Omega^\ell}
\frac{\partial u_\e}{\partial \nu} \cdot W_j^\ell (x/\e) \varphi\, d\sigma.
\endaligned
\end{equation}
It follows by Lemma \ref{lemma-cell} that 
$\| W_j^\ell (x/\e) \|_{L^2(\partial\Omega^\ell)} \le C$.
This, together with the first inequality in \eqref{4.4-0} with $m=1$, show that 
$$
\Big| \e^2 
\int_{\partial \Omega^\ell}
\frac{\partial u_\e}{\partial n} \cdot W_j^\ell (x/\e) \varphi\, d\sigma \Big|
\le C \e^2 \| (\nabla u_\e) \varphi \|_{L^2(\partial\Omega^\ell)} 
=O(\e^{1/2}).
$$
Hence, by \eqref{4.5-2},
\begin{equation}\label{4.5-3}
\aligned
 & \lim_{\e \to 0}
\e \int_{\Omega_\e^\ell}
\nabla u_\e \cdot \nabla W_j^\ell (x/\e) \varphi\, dx\\
& =\int_{\Omega^\ell} f \cdot K_j^\ell \varphi\, dx
+\int_{\Omega^\ell}
P K_j^\ell \cdot \nabla \varphi\, dx
-\lim_{\e\to 0}
\int_{\partial \Omega^\ell}
p_\e n  \cdot W_j^\ell (x/\e) \varphi\, d\sigma.
\endaligned
\end{equation}

Next, note that 
\begin{equation}\label{4.5.3-0}
\aligned
& \e^2 \int_{\Omega_\e^\ell} 
\nabla \left( W_j^\ell (x/\e) \right) \cdot \nabla (u_\e \varphi)\, dx\\
&=\e \int_{\Omega_\e^\ell} \nabla W_j^\ell (x/\e) \cdot  (\nabla u_\e) \varphi\, dx
+ \e \int_{\Omega_\e^\ell} \nabla W_j^\ell (x/\e) \cdot u_\e (\nabla \varphi)\, dx.
\endaligned
\end{equation}
Choose a cut-off function $\eta_\e$ such that
supp$(\eta_\e) \subset \{ x\in \R^d: \text{dist}(x, \partial \Omega^\ell) \le 2C \e \}$,
$ \eta_\e (x)=1$ if dist$(x, \partial \Omega^\ell) \le C\e$, and $|\nabla \eta_\e|
\le C \e^{-1}$.
Then
\begin{equation}\label{4.5-3-1}
\aligned
 &  \e^2 \int_{\Omega_\e^\ell} 
\nabla \left( W_j^\ell (x/\e) \right) \cdot \nabla (u_\e \varphi)\, dx\\
&= \e^2 \int_{\Omega_\e^\ell} 
\nabla \left( W_j^\ell (x/\e) \right) \cdot \nabla (u_\e (1-\eta_\e) \varphi)\, dx
+ \e^2 \int_{\Omega_\e^\ell} 
\nabla \left( W_j^\ell (x/\e) \right) \cdot \nabla (u_\e \eta_\e  \varphi)\, dx\\
& =J_1 + J_2.
\endaligned
\end{equation}
Using \eqref{4.5.3-0}, \eqref{4.5-3-1}, and 
$$
\aligned
|J_2|
 & \le C \e \left( 
\int_{\{ x\in \R^d: \, \text{dist}(x, \partial \Omega^\ell) \le C \e\} } |\nabla W_j^\ell (x/\e)|^2\, dx \right)^{1/2}
\left\{ \| \nabla u_\e \|_{L^2(\Omega)} + \e^{-1} \| u_\e\|_{L^2(\Omega)} \right\}\\
& \le C \e^{3/2} \left\{ \| \nabla u_\e \|_{L^2(\Omega)} + \e^{-1} \| u_\e\|_{L^2(\Omega)} \right\}\\
& \le \e^{1/2} C (f, h),
\endaligned
$$
we obtain
\begin{equation}\label{4.5-5}
\lim_{\e \to 0}
\e \int_{\Omega_\e^\ell} \nabla W_j^\ell (x/\e) \cdot  (\nabla u_\e) \varphi\, dx
=\lim_{\e \to 0} J_1
\end{equation}

To handle the term $J_1$, we use integration by parts as well as  the fact that
$$
-\e^2 \Delta \left( W_j^\ell (x/\e) \right)
+\nabla \left( \e \pi_j^\ell (x/\e) \right)
=e_j 
$$
in the set $\{ x\in \Omega_\e^\ell: \text{\rm dist}(x, \partial\Omega^\ell)\ge C \e \}$, to obtain
$$
\aligned
J_1
&=\int_{\Omega_\e^\ell}
\e \pi_j^\ell (x/\e) u_\e \cdot \nabla ( (1-\eta_\e) \varphi)\, dx
+\int_{\Omega^\ell}  e_j \cdot u_\e \varphi (1-\eta_\e)\, dx\\
&=J_{11} +J_{12},
\endaligned
$$
where we have used the fact div$(u_\e)=0$ in $\Omega_\e$.
Since
$$
\aligned
|J_{11}|
  & \le C  \left( 
\int_{\{ x\in \R^d: \, \text{dist}(x, \partial \Omega^\ell) \le C \e\} } | \pi_j^\ell (x/\e)|^2\, dx \right)^{1/2}
\| u_\e \|_{L^2(\Omega^\ell)} + C \e \| u_\e\|_{L^2(\Omega^\ell)}\\
& \le C \e^{1/2}  C (f, h),
\endaligned
$$
we see that
\begin{equation}\label{4.5-7}
\lim_{\e \to 0} J_1
=\lim_{\e \to 0} J_{12}
=\int_{\Omega^\ell}
e_j \cdot u \varphi \, dx.
\end{equation}
In view of \eqref{4.5-3}, \eqref{4.5-5} and \eqref{4.5-7}, we have proved that
\begin{equation}\label{4.5-7a}
\lim_{\e \to 0}
\int_{\partial\Omega^\ell}
p_\e n \cdot W_j^\ell (x/\e)\varphi\, d\sigma
=\int_{\Omega^\ell} f\cdot K_j^\ell \varphi\, dx
+ \int_{\Omega^\ell} P K_j^\ell \cdot \nabla \varphi\, dx
-\int_{\Omega^\ell} e_j \cdot u \varphi\, dx.
\end{equation}
Recall that $K^\ell=(K^\ell_{ij})$ is symmetric and 
by Lemma \ref{lemma-4.2},
$$
u= K^\ell (f-\nabla P) \quad \text{ in } \Omega^\ell.
$$
Thus, by \eqref{4.5-7a},
\begin{equation}\label{4.5-8}
\lim_{\e \to 0}
\int_{\partial\Omega^\ell}
p_\e n \cdot W_j^\ell (x/\e)\varphi\, d\sigma
=\int_{\partial\Omega^\ell} P^\ell ( n\cdot K_j^\ell ) \varphi \, d\sigma,
\end{equation}
where $P^\ell$ denotes the trace of $P$ on $\partial\Omega^\ell$.

Finally, we use Lemma \ref{skew-lemma} to obtain 
\begin{equation}\label{4.5-11}
n\cdot \left( W^\ell _j (x/\e) -K^\ell_j \right)
=\frac{\e}{2} 
\left( n_\beta \frac{\partial}{\partial x_\alpha}
-n_\alpha \frac{\partial}{\partial x_\beta}\right) 
\left( \phi^\ell_{\alpha \beta j} (x/\e)\right),
\end{equation}
where the repeated indices $\alpha$ and $\beta$ are summed from $1$ to $d$.
Since $n_\beta \frac{\partial}{\partial x_\alpha} -n_\alpha \frac{\partial}{\partial x_\beta}$
is a tangential derivative on $\partial\Omega^\ell$, we obtain
$$
\aligned
 & \Big| \int_{\partial \Omega^\ell}
p_\e n \cdot \left( W_j^\ell (x/\e) -K^\ell_j \right) \varphi \, d\sigma \Big|\\
&= \frac{\e}{2}
\Big|
\int_{\partial \Omega^\ell}
\phi^\ell_{\alpha \beta j} (x/\e)
\left( n_\beta \frac{\partial}{\partial x_\alpha} -n_\alpha \frac{\partial}{\partial x_\beta} \right) (p_\e \varphi)\, d\sigma \Big|\\
& \le C \e  \| \nabla (p_\e \varphi) \|_{L^2(\partial \Omega^\ell)}\\
&\le C(f, h) \e^{1/2},
\endaligned
$$
where we have used \eqref{skew-2} for the first  inequality and \eqref{4.4-0} for the last.
This, together with \eqref{4.5-8}, yields 
\begin{equation}\label{4.5-12}
\lim_{\e \to 0}
\int_{\partial\Omega^\ell}
p_\e (n \cdot K_j^\ell) \varphi\, d\sigma
=\int_{\partial\Omega^\ell} P^\ell ( n\cdot K_j^\ell ) \varphi \, d\sigma.
\end{equation}
Since the constant matrix $K^\ell=(K_{ij}^\ell)$ is invertible,
the  desired equation \eqref{4.5-0} follows readily  from \eqref{4.5-12}.
\end{proof}

We are now in a position to give the proof of Theorem \ref{H-thm}.

\begin{proof}[Proof of Theorem \ref{H-thm}]

We first prove Theorem \ref{H-thm} under the additional assumption $f\in C^\infty(\R^d; \R^d)$.
Let $\{ \e_k\}$ be a sequence such that $\e_k \to 0$,
$u_{\e_k} \to u$ weakly in $L^2(\Omega; \R^d)$ and $P_{\e_k} \to P$  strongly in $L^2(\Omega)$.
By Lemma \ref{lemma-4.2},  $P\in H^1(\Omega^\ell)$ and
$u=K^\ell (f-\nabla P)$ in $\Omega^\ell$  for $1\le \ell \le L$.
It  suffices to show that $P\in H^1(\Omega)$.
This would imply that $P$ is  a weak solution of the Neumann problem,
\begin{equation}\label{N-4}
\left\{
\aligned
\text{\rm div} (K (f-\nabla P))  & =0 & \quad & \text{ in } \Omega,\\
n \cdot K(f-\nabla P) & =n \cdot h& \quad & \text{ on } \partial\Omega.
\endaligned
\right.
\end{equation}
As a result, we may deduce that as $\e\to 0$,
$u_\e \to u_0$ weakly in $L^2(\Omega; \R^d)$ and
$P_\e -\fint_{\Omega} P_\e \to P_0$ strongly in $L^2(\Omega)$, where $u_0= K (f-\nabla P_0)$ in $\Omega$  and
$P_0$ is the unique weak solution of \eqref{N-4} with $\int_\Omega P_0\, dx=0$.

To prove $u\in H^1(\Omega)$, we use the assumption  $f\in C^\infty(\R^d; \R^d)$
and  Lemma \ref{lemma-4.5} to obtain
$$
\sum_{\ell=1}^L
\int_{\partial \Omega^\ell} n_j P^\ell \varphi\, d\sigma
=\lim_{k \to \infty}
\sum_{\ell=1}^L
\int_{\partial\Omega^\ell} n_j p_{\e_k} \varphi \, d\sigma,
$$
for any $\varphi \in C_0^\infty(\Omega)$ and $1\le j \le d$,
where $P^\ell$ denotes the trace of $P$, as a function in $H^1(\Omega^\ell)$, on $\partial\Omega^\ell$.
Since $p_\e$ is continuous in $\Omega_\e$,  we have
$$
\sum_{\ell=1}^L
\int_{\partial\Omega^\ell} n_j p_{\e} \varphi \, d\sigma=0.
$$
It follows that
$$
\sum_{\ell=1}^L
\int_{\partial \Omega^\ell} n_j P^\ell \varphi\, d\sigma=0
$$
for $1\le j \le d$ and for any $\varphi \in C_0^\infty(\Omega)$.
This, together with the fact that $P\in H^1(\Omega^\ell)$ for $1\le \ell \le L$, gives
$$
\aligned
\int_\Omega P\frac{\partial \varphi}{\partial x_j}\, dx
&=\sum_{\ell=1}^L  \int_{\Omega^\ell}  P\frac{\partial \varphi}{\partial x_j}\, dx\\
&=-\sum_{\ell=1}^L \int_{\Omega^\ell} \frac{\partial P}{\partial x_j} \varphi\, dx
+\sum_{\ell=1}^L \int_{\partial\Omega^\ell} n_j P^\ell \varphi\, d\sigma\\
& =-\sum_{\ell=1}^L \int_{\Omega^\ell} \frac{\partial P}{\partial x_j} \varphi\, dx.
\endaligned
$$
As a result, we obtain  $P\in H^1(\Omega)$.

In the general case $f\in L^2(\Omega; \R^d)$, we choose a sequence of functions $\{ f_m \}$ in $C^\infty(\R^d; \R^d)$
such that $\| f_m - f\|_{L^2(\Omega)} \to 0$ as $m \to \infty$.
Let $(u_{\e, m}, p_{\e, m})$ denote the weak solution of \eqref{D-4}
with $f_m$ in the place of $f$ and with $\int_{\Omega_\e} p_{\e, m} \, dx=0$.
By the energy estimates \eqref{3.2-0} and \eqref{cor-p-0} we obtain 
\begin{equation}
\| u_\e - u_{\e,m} \|_{L^2(\Omega)}
+ \| P_\e-P_{\e, m} \|_{L^2(\Omega)}
\le C \| f- f_m\|_{L^2(\Omega)},
\end{equation}
where $P_{\e, m}$ denotes the extension of $p_{\e, m}$,  defined by \eqref{ext-1}.
Let $u_{0, m}=K (f_m -\nabla P_{0, m})$,
where $P_{0, m}$ is the unique solution of \eqref{N-4} with $f_m$ in the place of $f$  and with 
$\int_\Omega P_{0, m} \, dx =0$.
Note that
$$
\aligned
\| P_\e -\fint_\Omega P_\e  - P_0\|_{L^2(\Omega)}
\le &  \| P_\e  -P_{\e, m} - \fint_{\Omega} (P_\e- P_{\e, m}) \|_{L^2(\Omega)}\\
& \qquad + \| P_{\e, m} - \fint_{\Omega} P_{\e, m} -P_{0,m} \|_{L^2(\Omega)} 
+ \| P_{0,m} -P_0\|_{L^2(\Omega)}\\
&\le C \| f-f_m\|_{L^2(\Omega)}
+   \| P_{\e, m} - \fint_{\Omega} P_{\e, m} -P_{0,m} \|_{L^2(\Omega)}.
\endaligned
$$
Since $P_{\e, m} -\fint_\Omega P_{\e, m} \to P_{0, m}$ in  $L^2(\Omega)$, as $\e \to 0$,
we see that
$$
\limsup_{\e \to 0} 
\| P_\e -\fint_\Omega P_\e  - P_0\|_{L^2(\Omega)}
\le C \| f-f_m \|_{L^2(\Omega)}.
$$
By letting $m\to\infty$, we obtain $P_\e -\fint_\Omega P_\e \to P_0$ in $L^2(\Omega)$,
as $\e \to 0$.

Finally, let $v \in L^2(\Omega; \R^d)$.
Note that
$$
\aligned
 & \Big| \int_\Omega (u_\e -u_0) v\, dx \Big |\\
& \le \Big| \int_{\Omega} (u_\e - u_{\e, m} ) v\, dx\Big|
+ \Big| \int_\Omega (u_{\e, m} -u_{0, m }) v \, dx\Big|
+\Big| \int_\Omega (u_{0, m} - u_0) v \, dx\Big|\\
& \le \| u_\e -u_{\e, m} \|_{L^2(\Omega)} \| v \|_{L^2(\Omega)}
+  \Big| \int_\Omega (u_{\e, m} -u_{0, m }) v \, dx\Big|
+ \| u_{0, m} - u_0 \|_{L^2(\Omega)} \| v\|_{L^2(\Omega)}\\
&\le  C \| f-f_m\|_{L^2(\Omega)} \| v\|_{L^2(\Omega)}
+ \Big| \int_\Omega (u_{\e, m} -u_{0, m }) v \, dx\Big|.
\endaligned
$$
By letting $\e \to 0$ and then $m \to \infty$, we see that $u_\e \to u_0$ weakly in $L^2(\Omega; \R^d)$.
\end{proof}


\section{Convergence rates and proof of Theorem \ref{main-thm-2}}\label{section-C}

Throughout the rest of this paper, unless indicated otherwise,
 we will assume that $\Omega^\ell, 1\le \ell \le L$,  are $C^{2, 1/2}$ domains satisfying
the interface condition \eqref{condition-1}. 
Given $f\in L^2 (\Omega; \R^d)$, let $P_0\in H^1(\Omega)$ be the weak solution of
\begin{equation}\label{P-0}
\left\{
\aligned
-\text{\rm div} \left( K (f-\nabla P_0) \right) & =0 & \quad & \text{ in } \Omega,\\
n\cdot K(f-\nabla P_0) & =0 & \quad & \text{ on } \partial\Omega,
\endaligned
\right.
\end{equation}
with $\int_\Omega P_0\, dx=0$, where  the coefficient matrix $K$ is given by \eqref{K-0}.
Since the interface $\Sigma$ and $\partial\Omega$ are of $C^{2, 1/2}$, it follows from \cite[Theorem 1.1]{Zhuge} that
\begin{equation}\label{P-reg}
\aligned
\| \nabla P_0 \|_{C^{\alpha}(\Omega)} &  \le C \| f\|_{C^{ \alpha}(\Omega)},\\
\| \nabla P_0\|_{C^{1, \beta}(\Omega)} & \le C \| f \|_{C^{1, \beta}(\Omega)},
\endaligned
\end{equation}
for $0< \alpha< 1$ and $0< \beta \le 1/2$.

Let
\begin{equation}\label{V01}
V_\e  (x)= \sum_{\ell=1}^L W^\ell (x/\e) (f-\nabla P_0) \chi_{\Omega^\ell} \quad \text{ in } \Omega, 
\end{equation}
where the 1-periodic matrix $W^\ell (y)$ is defined by \eqref{cor-1} .
Note that $V_\e =0$ in $\Gamma_\e$.
For each $\ell$, using
\begin{equation}\label{V10}
-\e^2 \Delta \left\{ W_j^\ell (x/\e)\right\}  +\nabla \left \{ \e \pi_j^\ell (x/\e) \right\}
= e_j \quad \text{ in }  \bigcup_{z\in \mathbb{Z}^d} \e (z+{Y^\ell_f}),
\end{equation}
one may show that for any $\psi\in H^1(\Omega^\ell_\e; \R^d)$ with $\psi=0$ on $ \Gamma^\ell_\e$,
\begin{equation}\label{V11}
\aligned
 & \Big|  \e \int_{\Omega^\ell_\e}
  \nabla W_j^\ell (x/\e) \cdot \nabla \psi \, dx-\e \int_{\Omega_\e^\ell} \pi_j ^\ell (x/\e) \, \text{div}(\psi)\, dx
-\int_{\Omega_\e^\ell} \psi_j  \, dx \Big|\\
&   \le C \e^{3/2} \|\nabla \psi \|_{L^2(\Omega^\ell_\e)}.
\endaligned
\end{equation}
To see \eqref{V11},  let
$$
\mathcal{O}_\e^\ell = \bigcup_z  \e (z+ Y_f^\ell),
$$
where  $z\in \mathbb{Z}^d$ and the union is taken over those $z$'s for which  $\e (z + Y)\subset \Omega^\ell$.
Using  $|\Omega_\e^\ell \setminus \mathcal{O}_\e^\ell| \le C \e^{1/2}$ and
$\|\psi \|_{L^2(\Omega^\ell_\e)} \le C \e \|\nabla \psi\|_{L^2(\Omega^\ell_\e)}$, one may show that
each integral  in the left-hand side  of \eqref{V11}, with $\Omega_\e^\ell \setminus \mathcal{O}^\ell_\e$ in the place of
$\Omega_\e^\ell$, 
is  bounded by the right-hand side of \eqref{V11}.
By using integration by parts and \eqref{V10}, it follows that the left-hand side of \eqref{V11} with 
$\mathcal{O}_\e^\ell$ in the place of $\Omega_\e^\ell$ is bounded by 
$$
\aligned
& C \e \left(\int_{\partial \mathcal{O}_\e^\ell} \left( |\nabla W^\ell (x/\e)| +|\pi^\ell (x/\e)|\right)^2 \, d\sigma \right)^{1/2}
\left(\int_{\partial \mathcal{O}_\e^\ell} |\psi|^2 \, d\sigma \right)^{1/2}\\
& \le C \e^{3/2} \|\nabla \psi\|_{L^2(\Omega^\ell_\e)},
\endaligned
$$
where we have used \eqref{cell-2} and  the observation,
$$
\aligned
\|\psi\|_{L^2(\partial \mathcal{O}_\e^\ell)}
 & \le C \e^{-1/2} \| \psi \|_{L^2(\Omega_\e^\ell)}
+ C \e^{1/2} \|\nabla \psi \|_{L^2(\Omega_\e^\ell)}\\
& \le C \e^{1/2} \| \nabla \psi \|_{L^2(\Omega_\e^\ell)}.
\endaligned
$$
From \eqref{V11} we deduce further  that
\begin{equation}\label{V11a}
\aligned
 & \Big|  \e \int_{\Omega^\ell_\e}
  \nabla W_j^\ell (x/\e) \cdot \nabla \psi \, dx
-\int_{\Omega_\e^\ell} \psi_j  \, dx \Big|\\
& \qquad  \le C \e^{1/2} \left\{ \e \|\nabla \psi \|_{L^2(\Omega^\ell_\e)}
+ \e^{1/2} \|\text{\rm div}(\psi) \|_{L^2(\Omega^\ell_\e)}\right\}
\endaligned
\end{equation}
for any $\psi\in H^1(\Omega^\ell_\e; \R^d)$ with $\psi=0$ on $\Gamma^\ell_\e$.

\begin{thm}\label{theorem-V0}
Let $(u_\e, p_\e)\in H_0^1(\Omega_\e; \R^d)\times L^2_0(\Omega_\e)$ be a weak solution of \eqref{Stokes-1}.
Let $V_\e$ be given by \eqref{V01}.
Then
\begin{equation}\label{V100}
\aligned
& \Big|\e^2\sum_{\ell=1}^L \int_{\Omega^\ell_\e}  (\nabla u_\e-\nabla V_\e) \cdot \nabla \psi\, dx 
-\int_{\Omega_\e} 
( p_\e -P_0 ) \, \text{\rm div} (\psi)\, dx
 \Big|\\
&\le C \e^{1/2} \| f \|_{C^{1, 1/2}(\Omega)} \left\{ \e \|\nabla \psi \|_{L^2(\Omega_\e)}
+ \e^{1/2}  \| \text{\rm div}(\psi)\|_{L^2(\Omega_\e)} \right\},
\endaligned
\end{equation}
for any $\psi \in H_0^1(\Omega_\e; \R^d) $. 
\end{thm}

\begin{proof}
We apply \eqref{V11a} with $\psi (f_j -\frac{\partial P_0}{\partial x_j})$ in the place of $\psi$.
Using 
$$
\aligned
 & | \e^2 \nabla V_\e \cdot \nabla \psi - \e \nabla W^\ell (x/\e) \cdot \nabla 
\left( \psi (f-\nabla P_0) \right)|\\
& \le C \left\{ \e^2 |W^\ell (x/\e)| |\nabla \psi|
+ C \e |\nabla W^\ell(x/\e)|  |\psi|\right\} | \nabla (f-\nabla P_0)|
\endaligned
$$
in $\Omega^\ell_\e$, we obtain 
$$
\aligned
& \Big|
\e^2 \int_{\Omega_\e^\ell}
\nabla V _\e \cdot \nabla \psi\, dx
-\int_{\Omega_\e^\ell} (f-\nabla P_0)\cdot \psi\, dx \Big|\\
& \le C\e^{3/2}  \left( \| f\|_\infty + \| \nabla f\|_\infty
+ \| \nabla P_0\|_\infty + \|\nabla^2 P_0 \|_\infty \right)
  \|\nabla \psi\|_{L^2(\Omega^\ell_\e)}\\
  &\qquad
  + C \e  (\| f\|_\infty+ \|\nabla P_0 \|_\infty) \|\text{\rm div}(\psi)\|_{L^2(\Omega_\e^\ell)}.
\endaligned
$$
This, together with 
$$
\int_{\Omega_\e} (f-\nabla P_0) \cdot \psi\, dx
=\e^2 \int_{\Omega_\e}\nabla u_\e\cdot \nabla \psi\, dx
-\int_{\Omega_\e} (p_\e -P_0) \, \text{\rm div}(\psi)\, dx,
$$
gives \eqref{V100}.
\end{proof}

Let
\begin{equation}\label{U}
U_\e=V_\e + \Phi_\e,
\end{equation}
where $\Phi_\e$ is a corrector to be constructed so that 
 $U_\e \in H_0^1(\Omega_\e; \R^d)$,
\begin{equation}\label{div-e}
\|\text{\rm div} (U_\e) \|_{L^2(\Omega_\e)}
\le C  \e^{1/2} \| f\|_{C^{1, 1/2} (\Omega)},
\end{equation}
and that 
\begin{equation}\label{P-est}
\e \|\nabla \Phi_\e\|_{L^2(\Omega^\ell_\e)}
\le C \e^{1/2} \| f\|_{C^{1, 1/2}(\Omega)}
\end{equation}
for $1\le \ell \le L$.

Assuming that  such corrector $\Phi_\e$ exists, we give the proof of Theorem \ref{main-thm-2}.

\begin{proof}[Proof of Theorem \ref{main-thm-2}]

By letting $\psi=u_\e-U_\e=u_\e -V_\e -\Phi_\e \in H_0^1(\Omega_\e; \R^d)$ in \eqref{V100},
we obtain 
$$
\aligned
  & \e^2 \| \nabla u_\e-\nabla V_\e\|^2_{L^2(\Omega_\e)}\\
 & \le \e^2 \|\nabla u_\e-\nabla V_\e\|_{L^2(\Omega_\e)} \| \nabla \Phi_\e\|_{L^2(\Omega_\e)}
 + \| p_\e -P_0-\beta \|_{L^2(\Omega_\e)}
 \| \text{\rm div} (U_\e)  \|_{L^2(\Omega_\e)}\\
& \qquad+ C \e^{1/2}
\ \| f\|_{C^{1, 1/2}(\Omega)} 
\left\{ \e \| \nabla u_\e -\nabla V_\e \|_{L^2(\Omega_\e)}
+\e \|\nabla \Phi_\e\|_{L^2(\Omega_\e)}
+ \e^{1/2} \| \text{\rm div} (U_\e) \|_{L^2(\Omega_\e)} \right\}\\
& \le C \e^{3/2} \|  f\|_{C^{1, 1/2} (\Omega)} \| \nabla u_\e-\nabla V_\e\|_{L^2(\Omega_\e)}\\
&\qquad
+ C \e^{1/2} \| f \|_{C^{1, 1/2} (\Omega)} \| p_\e -P_0-\beta\|_{L^2(\Omega_\e)}
+ C \e \| f\|_{C^{1, 1/2} (\Omega)}^2,
\endaligned
$$
for any $\beta \in \R$, where we have used \eqref{div-e} and \eqref{P-est} for the last inequality.
By the Cauchy inequality, this implies that
\begin{equation}\label{pp-e0}
\e^2 \| \nabla u_\e-\nabla V_\e\|^2_{L^2(\Omega_\e)}
\le  C \e \| f\|^2_{C^{1, 1/2}(\Omega)}
+ C \e^{1/2} \| f\|_{C^{1, 1/2}(\Omega)} \| p_\e-P_0 -\beta \|_{L^2(\Omega_\e)}.
\end{equation}
We should point out that both $V_\e$ and $\Phi_\e$ are not in $H^1(\Omega_\e; \R^d)$.
In the estimates above (and thereafter)  we have used the convention that
$$
\|\ \nabla \psi  \|_{L^2(\Omega_\e)}
=\left(\sum_{\ell=1}^L  \| \nabla \psi \|^2_{L^2(\Omega_\e^\ell)} \right)^{1/2},
$$
where $\psi \in H^1(\Omega_\e^\ell)$ for $1\le \ell \le L$.

Next, we choose  $\beta=\fint_{\Omega_\e} (p_\e-P_0)$. By Lemma \ref{lemma-div},
 there exists  $v_\e\in H^1_0 (\Omega_\e; \R^d)$ such that
$$
\aligned
\text{\rm div}(v_\e)  & = p_\e-P_0 -\beta \quad \text{ in } \Omega_\e, \\
\e \| \nabla v_\e \|_{L^2(\Omega_\e)} &  \le C  \| p_\e -P_0 -\beta\|_{L^2(\Omega_\e)}.
\endaligned
$$
By letting $\psi_\e=v_\e$ in \eqref{V100}, we obtain 
\begin{equation}\label{pp-e}
\| p_\e-P_0 -\beta\|_{L^2(\Omega_\e)}
\le C \e \| \nabla u_\e -\nabla V_\e \|_{L^2(\Omega_\e)}
+C \e^{1/2} \| f\|_{C^{1, 1/2}(\Omega)}.
\end{equation}
By combining \eqref{pp-e0} with \eqref{pp-e}, it is not hard to see that
\begin{equation}\label{pp-ee}
\e\| \nabla u_\e -\nabla V_\e \|_{L^2(\Omega_\e)} + \| p_\e-P_0 -\beta\|_{L^2(\Omega_\e)}
\le  C \e^{1/2} \| f\|_{C^{1, 1/2} (\Omega)}.
\end{equation}
This, together  with $\| u_\e -V_\e\|_{L^2(\Omega_\e^\ell)} \le C \e \| \nabla u_\e -\nabla V_\e \|_{L^2(\Omega^\ell_\e)}$,  
gives the bound for the first term in  \eqref{main-1}.
Also, note that 
$$
\| \e \nabla V_\e  -\nabla W^\ell (x/\e) (f-\nabla P_0) \|_{L^2(\Omega_\e^\ell)}
\le C \e \|\nabla (f-\nabla P_0)\|_\infty.
$$
Thus,
$$
\| \e \nabla u_\e - \nabla W^\ell (x/\e) (f-\nabla P_0)\|_{L^2(\Omega_\e^\ell)}
\le C \e^{1/2} \| f\|_{C^{1, 1/2}(\Omega)}.
$$

Finally, to estimate the pressure, we let $Q_\e$  be the extension of $(P_0+\beta)|_{\Omega_\e}$  to $\Omega$, using 
the formula in \eqref{ext-1}.
Note that
$$
\| Q_\e- (P_0 +\beta)\|^2_{L^2(\Omega)}
=\sum_{\ell, z}  \int_{\e (Y_s^\ell +z)}
\Big|P_0 - \fint_{\e (Y_f^\ell +z)}P_0 \Big|^2\, dx,
$$
where the sum is taken over those $(\ell, z)$'s for which $ z \in \mathbb{Z}^d$ and $\e (Y+z) \subset \Omega^\ell$.
It follows that
$$
\aligned
\| Q_\e - (P_0+\beta)\|_{L^2(\Omega)}
 &\le C \e \| \nabla P_0 \|_{L^\infty(\Omega)}\\
&\le C \e \| f\|_{C^{1, 1/2}(\Omega)}.
\endaligned
$$
As a result, by \eqref{pp-ee}, we obtain 
$$
\aligned
\| P_\e - P_0 -\beta\|_{L^2(\Omega)}
& \le \| P_\e - Q_\e \|_{L^2(\Omega)}
+ \| Q_\e - (P_0 +\beta)\|_{L^2(\Omega)}\\
&\le  C \| p_\e - P_0 -\beta\|_{L^2(\Omega_\e)}
+  C \e \| f\|_{C^{1, 1/2} (\Omega)}\\
& \le C \e^{1/2} \| f \|_{C^{1, 1/2}(\Omega)},
\endaligned
$$
where $\beta =-\fint_{\Omega_\e} P_0 $.
Clearly, we may replace $\beta$ by $\fint_\Omega (P_\e -P_0)=\fint_\Omega P_\e$.
This gives the bound for the second term in \eqref{main-1}.
\end{proof}

To complete the proof of Theorem \ref{main-thm-2},
it remains to construct a corrector $\Phi_\e$ such that $V_\e +\Phi_\e \in H_0^1(\Omega_\e; \R^d)$ and
\eqref{div-e}-\eqref{P-est} hold.
This will be done in the next three sections.
More precisely, we   let
\begin{equation}
\Phi_\e = \Phi_\e^{(1)} + \Phi_\e^{(2)} + \Phi_\e^{(3)},
 \end{equation}
where  $\Phi_\e^{(1)}$ is a corrector for the divergence operator with the properties that
\begin{equation}\label{PP-1}
\left\{
\aligned
&  \Phi_\e^{(1)}   \in H_0^1 (\Omega_\e; \R^d),\\
&   \e \| \nabla \Phi_\e^{(1)} \|_{L^2(\Omega_\e)}   \le C \e^{1/2} \| f\|_{C^{1, 1/2} (\Omega)}, \\
 &\| \text{\rm div}(\Phi_\e^{(1)} +V_\e )  \|_{L^2(\Omega_\e^\ell)}
  \le C \e^{1/2}\| f\|_{C^{1, 1/2}(\Omega)},
\endaligned
\right.
\end{equation}
$\Phi_\e^{(2)} $ is a corrector  for the boundary data of $V_\e$  on $\partial\Omega$ with the properties that
\begin{equation}\label{PP-2}
\left\{
\aligned
  & \Phi_\e^{(2)}   \in H^1 (\Omega_\e; \R^d) \quad \text{ and } \quad \Phi^{(2)}_\e =0 \quad \text{ on } \Gamma_\e,\\
  &  \Phi_\e^{(2)} +V_\e =0\quad  \text{ on } \partial\Omega,\\
  &  \e \| \nabla \Phi_\e^{(2)} \|_{L^2(\Omega_\e)}  + \| \text{\rm div} (\Phi_\e^{(2)}) \|_{L^2(\Omega_\e)}
    \le C \e^{1/2} \| f\|_{C^{1, 1/2} (\Omega)}, 
\endaligned
\right.
\end{equation}
and 
$\Phi_\e^{(3)} $ is a corrector  for the interface $\Sigma$ with the properties that
\begin{equation}\label{PP-3}
\left\{
\aligned
& \Phi_\e^{(3)}   \in H^1 (\Omega^\ell_\e;  \R^d) \quad \text{ and }  \quad \Phi_\e^{(3)}= 0 \quad \text{ on } \partial \Omega_\e,\\
& V_\e + \Phi_\e^{(3)}\in H^1(\Omega_\e; \R^d),\\
  & \e \| \nabla \Phi_\e^{(3)} \|_{L^2(\Omega^\ell_\e)}  
  + \| \text{\rm div} (\Phi_\e^{(3)}) \|_{L^2(\Omega_\e^\ell)} \le C \e^{1/2} \| f\|_{C^{1, 1/2} (\Omega)},
\endaligned
\right.
\end{equation}
for $1\le \ell \le L$.
It is not hard to verify that the desired property  $V_\e +\Phi_\e \in H_0^1(\Omega_\e; \R^d)$
as well as the estimates \eqref{div-e}-\eqref{P-est}
follows from \eqref{PP-1}-\eqref{PP-3}.


\section{Correctors for the divergence operator}\label{section-D}

Let $V_\e$ be given by \eqref{V01}. Note that since div$(W_j^\ell(x/\e))=0$ in $\R^d$, 
\begin{equation}\label{div-1}
\text{\rm div} (V_\e)
=W^\ell  (x/\e) \nabla (f-\nabla P_0) \quad \text{ in } \Omega^\ell_\e.
\end{equation}
In this section we construct a corrector $\Phi_\e^{(1)}$ that satisfies \eqref{PP-1}.
The approach is similar to that used  in \cite{Mik-1996, Shen-Darcy-1}.

For $1\le \ell \le L$ and $1\le \, i, j \le d$, let $\Theta_{ij}^\ell= (\Theta^\ell_{1ij}, \dots, \Theta^\ell_{dij}) $ be a 1-periodic function in $H^1_{loc}(\R^d; \R^d)$
such that
\begin{equation}\label{Theta}
\left\{
\aligned
\text{\rm div} (\Theta^\ell_{ij}) & = -W^\ell_{ij}  + | Y_f|^{-1} K^\ell_{ij}  & \quad & \text{ in } Y_f,\\
\Theta_{ij}^\ell & =0 & \quad & \text{ in } Y_s.
\endaligned
\right.
\end{equation}
 Fix $\varphi \in C_0^\infty(B(0, 1/8))$ such that $\varphi\ge 0$ and $\int_{\R^d} \varphi\, dx =1$.
 Define
 \begin{equation}
S_\e (\psi) (x) =\psi * \varphi_\e (x)=\int_{\R^d} \psi (y) \varphi_\e (x-y)\, dy,
\end{equation}
where $\varphi_\e (x)=\e^{-d} \varphi (x)$.
Let $\Phi_\e^{(1)} =( \Phi_{\e, 1}^{(1)}, \dots, \Phi_{\e, d}^{(1)} )$, where, for $x\in \Omega_\e^\ell$, 
\begin{equation}\label{Phi-1}
\Phi_{\e, k} ^{(1)} (x)
=\e \eta^\ell_\e (x)  \Theta^\ell_{kij} (x/\e) \frac{\partial}{\partial x_i} 
S_\e  \left( f_j -\frac{\partial P_0}{\partial x_j} \right),
\end{equation}
and $P_0$ is the solution of \eqref{P-0}.
The function $\eta_\e^\ell $ in \eqref{Phi-1} is a cut-off function in $C_0^\infty(\Omega^\ell)$ with the properties that
$|\nabla \eta_\e^\ell|\le C\e^{-1}$ and 
$$
\left\{
\aligned
 & \eta_\e^\ell (x) =0 & \quad & \text{ if dist} (x, \partial\Omega^\ell)\le 2d \e,\\
 & \eta^\ell_\e (x)=1 & \quad & \text{ if  } x\in \Omega^\ell \text{ and dist} (x, \partial \Omega^\ell)\ge 3d \e.
 \endaligned
 \right.
 $$
 As  a result, $\Phi_\e^{(1)}$ vanishes near $\partial \Omega^\ell$.
 
 \begin{thm}\label{theorem-PP1}
 Let $\Phi_\e^{(1)}$ be defined by \eqref{Phi-1}.
 Then \eqref{PP-1}  holds.
 \end{thm}
 
 \begin{proof}
 Clearly, $\Phi_\e^{(1)} \in H_0^1(\Omega_\e; \R^d)$.
 Note that 
 $$
 \aligned
  \|\nabla \Phi_\e^{(1)}\|_{L^2(\Omega_\e^\ell)}
 &  \le  C \e^{1/2} \|\nabla S_\e(f-\nabla P_0) \|_{L^\infty (N_{3d\e}\setminus N_{2d\e})}
 + C  \|\nabla S_\e (f-\nabla P_0)\|_{L^\infty(\Omega^\ell \setminus N_{2d\e})}\\
&\qquad\qquad
 + C \e \|\nabla^2 S_\e (f-\nabla P_0)\|_{L^\infty(\Omega^\ell \setminus N_{2d\e})},
\endaligned
 $$
 where $N_r =\{ x\in \Omega^\ell: \text{\rm dist}(x, \partial \Omega^\ell) < r \}$.
 This, together with the observation that $\nabla S_\e (\psi)=S_\e( \nabla \psi)$ and 
 $$
|S_\e (\psi)(x)| + \e  |\nabla S_\e (\psi)(x)| \le C  \fint_{B(x, \e/8)} |\psi|,
 $$
 yields
 $$
 \aligned
 \e \|\nabla \Phi_\e^{(1)} \|_{L^2(\Omega_\e^\ell)}
  & \le C \e \|\nabla (f-\nabla P_0) \|_{L^\infty(\Omega^\ell)}\\
 & \le 
  C \e \| f \|_{C^{1, 1/2}(\Omega)}.
 \endaligned
 $$
 
Next,  note that in $\Omega^\ell_\e$,
$$
\aligned 
\text{\rm div} (\Phi_\e^{(1)})
 & =\e (\nabla \eta_\e^\ell) \Theta^\ell (x/\e) \nabla S_\e (f-\nabla P_0)
- \eta_\e^\ell W^\ell (x/\e) \nabla S_\e (f-\nabla P_0)\\
& \qquad\qquad
+\e \eta_\e^\ell \Theta^\ell (x/\e) \nabla^2 S_\e (f-\nabla P_0),
\endaligned
$$
where we have used the fact that  $\text{\rm div} (K^\ell (f-\nabla P_0)) =0$ in $\Omega_\e^\ell$.
It follows that
$$
\aligned
 &\| \text{\rm div} (\Phi_\e^{(1)}) + W^\ell (x/\e) \nabla (f-\nabla P_0) \|_{L^2(\Omega_\e^\ell)}\\
& \le C \e^{1/2} \|\nabla (f-\nabla P_0)\|_{L^\infty(\Omega^\ell)}
+ \| W^\ell (x/\e) \left\{ \nabla (f-\nabla P_0) -\eta_\e^\ell \nabla S_\e (f-\nabla P_0)\right\} \|_{L^2(\Omega_\e^\ell)}\\
&\qquad\qquad
+ C \e  \|\nabla^2 S_\e (f-\nabla P_0) \|_{L^\infty(\Omega_\e^\ell \setminus N_{2d\e}) }\\
& \le C \e^{1/2} \|\nabla (f-\nabla P_0)\|_{L^\infty(\Omega^\ell)}
+ \| \nabla (f-\nabla P_0) -\nabla S_\e (f-\nabla P_0) \|_{L^\infty(\Omega^\ell\setminus N_{2d\e})}\\
&\qquad\qquad
+ C \e  \|\nabla^2 S_\e (f-\nabla P_0) \|_{L^\infty(\Omega^\ell \setminus N_{2d\e}) }\\
&\le 
 C \e^{1/2}  \|\nabla (f-\nabla P_0) \|_{C^{1/2} (\Omega^\ell )}\\
& \le C \e^{1/2} \| f\|_{C^{1, 1/2}(\Omega)},
\endaligned
$$
where we have used \eqref{P-reg} for the last inequality.
In the third inequality above we also used the observation that
$$
\nabla S_\e (\psi) (x)=
-\int_{\R^d} \left (\psi(x-y) -\psi (x)\right) \nabla_y (\varphi_\e (y) ) \, dy, 
$$
which gives
$$
|\nabla S_\e (\psi) (x)|
\le C  \e^{\alpha -1}\| \psi \|_{C^{0, \alpha} (B(x, \e))}.
$$
This completes the proof of \eqref{PP-1}.
\end{proof}



\section{Boundary correctors}\label{section-5}

To construct the boundary corrector $\Phi_\e^{(2)}$, we 
consider the Dirichlet problem,
\begin{equation}\label{D-5}
\left\{
\aligned
-\e^2 \Delta u_\e +\nabla p_\e & =0 & \quad & \text{ in } \Omega_\e\\
\text{\rm div}(u_\e) & =\gamma & \quad & \text{ in } \Omega_\e,\\
u_\e & =0 & \quad & \text{ on } \Gamma_\e,\\
 u_\e & =h & \quad & \text{ on } \partial \Omega,
 \endaligned
 \right.
 \end{equation}
 where $\Omega_\e$ is given by \eqref{domain-e} and 
 \begin{equation}\label{com-5}
 \gamma =\frac{1}{|\Omega_\e|} \int_{\partial \Omega} h \cdot n \, d\sigma.
 \end{equation}
Let $\Phi_\e^{(2)} \in H^1(\Omega_\e; \R^d)$ be the solution of \eqref{D-5}  with boundary value,
\begin{equation} \label{5.7-1}
h= -V_\e \quad \text{ on } \partial \Omega,
\end{equation}
where $V_\e$ is given by \eqref{V01}. Thus, if $\partial \Omega \cap \partial \Omega^\ell\neq \emptyset$ for 
some $1\le \ell \le L$,
\begin{equation}\label{5.7-2}
\Phi_\e^{(2)}
=- W^\ell (x/\e) (f-\nabla P_0) \quad \text{ on } \partial\Omega \cap \partial \Omega^\ell.
\end{equation}

\begin{thm}\label{thm-Phi-2}
Let $\Phi_\e^{(2)}$ be defined as above. Then $\Phi_\e^{(2)}$ satisfies \eqref{PP-2}.
\end{thm}

To show Theorem \ref{thm-Phi-2}, we first prove some general results, which will be  used also in the construction of 
correctors for the interface.

\begin{thm}\label{thm-t}
Let $\Omega$ be a bounded Lipschitz domain in $\R^d$, $d\ge 2$.
Assume that $\Omega^\ell$ and $Y_s^\ell$  with $1\le \ell \le L$
are subdomains of $\Omega$ and $Y$, respectively,
with Lipschitz boundaries.
Let $(u_\e, p_\e)$ be a weak solution in $H^1(\Omega_\e; \R^d) \times L_0^2(\Omega_\e)$  of \eqref{D-5}, where
$h\in H^1(\partial\Omega; \R^d)$ and 
\begin{equation} \label{t-condition}
u\cdot n =0 \quad \text{ on } \partial\Omega.
\end{equation}
Then
\begin{equation}\label{5.1-0}
\e \| \nabla u_\e\|_{L^2(\Omega_\e)}
+ \| u_\e\|_{L^2(\Omega_\e)}
+ \| p_\e \|_{L^2(\Omega_\e)}
\le C \sqrt{\e} 
\left\{
\| h \|_{L^2(\partial\Omega)}
+\e \|\nabla_{\tan} h \|_{L^2(\partial \Omega)} \right\},
\end{equation}
where $\nabla_{\tan} h$ denotes the tangential gradient of $h$ on $\partial\Omega$.
\end{thm}

\begin{proof}
This theorem was proved in \cite[Theorem 4.1]{Shen-Darcy-1} for the case $L=1$.
The proof only uses the energy estimate \eqref{3.2-0} and the fact that 
$$
-\e^2 \Delta u_\e +\nabla p_\e = 0 \quad \text{ and } \quad \text{\rm div}(u_\e)=0
$$
in the set $\{ x\in \Omega:\ \text{\rm dist} (x, \partial\Omega)< c\, \e \}$.
As a result, the same proof works equally well for the case $L\ge 2$.
We mention that the argument relies on  the Rellich estimates in \cite{FKV} for the Stokes equations in Lipschitz domains. 
The  condition \eqref{t-condition} allows us to drop the pressure $p_\e$ term in the conormal
derivative  $\partial u_\e/{\partial \nu} $ for $u_\e$ on $\partial\Omega$. 
We omit the details.
\end{proof}

In the next theorem we consider the case where 
\begin{equation}\label{N-condition}
h\cdot n =\e  \ (\nabla_{\tan}  \phi_\e ) \cdot g )   \quad \text{ on } \partial \Omega.
\end{equation}
By using integration by parts on $\partial\Omega$,  we see that
\begin{equation}\label{gamma-e}
\aligned
|\gamma|  &\le C \Big|\int_{\partial\Omega} h\cdot n\,  d \sigma \Big|\\
 &\le C\e  \| \phi_\e \nabla_{\tan} g \|_{L^2(\partial\Omega)}.
\endaligned
\end{equation}

\begin{thm}\label{thm-N}
Let $\Omega$ be a bounded $C^{2, \alpha}$ domain in $\R^d$, $d\ge 2$.
Let $(u_\e, p_\e)$ be a weak solution in $H^1(\Omega_\e; \R^d)\times L_0^2(\Omega)$ of \eqref{D-5},
 where $h\in H^1(\partial\Omega)$ and $h\cdot n$  is given \eqref{N-condition}.
Then
\begin{equation}\label{5.2-0}
\aligned
& \e \|\nabla u_\e\|_{L^2(\Omega_\e)}
+\| u_\e \|_{L^2(\Omega_\e)}
+ \|p_\e\|_{L^2(\Omega_\e)}\\
& \qquad
\le C \sqrt{\e}
\left\{ \| h\|_{L^2(\partial \Omega)} +  \e \| \nabla_{\tan} h \|_{L^2(\partial \Omega)}
+\| \phi_\e g \|_{L^2(\partial\Omega)}
+ \e^{1/2} \| \phi_\e \nabla_{\tan}  g \|_{L^2(\partial\Omega)}
\right\}.
\endaligned
\end{equation}
\end{thm}

\begin{proof}
A version of this theorem was proved in \cite[Theorem 5.1] {Shen-Darcy-1} for the case $L=1$.
We give the proof for the general case, using a somewhat different argument.

We first note that by writing
$$
h= ( h - (h\cdot n) n) + (h\cdot n) n
$$
and applying Theorem \ref{thm-t} to the solution of \eqref{D-5}  with boundary data $h- (h\cdot n) n$, 
we may reduce the problem to case where  $h=(h\cdot n) n$ on $\partial\Omega$.

Next, by the energy estimate \eqref{3.1-0} and \eqref{gamma-e},
\begin{equation}\label{5.2-1}
\aligned
& \e \|\nabla u_\e\|_{L^2(\Omega_\e)}
+\| u_\e \|_{L^2(\Omega_\e)}
+ \|p_\e\|_{L^2(\Omega_\e)}\\
& \qquad \le C \left\{ \| H \|_{L^2(\Omega)}
+ \|\text{\rm div}(H)\|_{L^2(\Omega)}
+ \e \| \nabla H \|_{L^2(\Omega)}  + \e \|\phi_\e \nabla_{\tan} g \|_{L^2(\partial\Omega)} \right\},
\endaligned
\end{equation}
where $H$ is any function in $H^1(\Omega; \R^d)$ with  $H=h$ on $\partial\Omega$.
We choose $H=H_1+ \gamma (x-x_0) /d $, where $x_0 \in \Omega$ and $H_1$ is  the weak solution of
$$
-\Delta H_1 +\nabla q =0 \quad \text{ and } \quad \text{\rm div}(H_1)=0 \quad \text{ in } \Omega,
$$
with the boundary value $H_1=h- \gamma (x-x_0)/d$ on $\partial\Omega$.
It follows that 
\begin{equation}\label{5.2-1a}
\aligned
 & \e \|\nabla u_\e\|_{L^2(\Omega)}
+\| u_\e \|_{L^2(\partial\Omega)}
+ \|p_\e\|_{L^2(\Omega)}\\
& \qquad
\le C \left\{ \| H _1\|_{L^2(\Omega)}
+ \e \| \nabla H_1 \|_{L^2(\Omega)} + \e \| \phi_\e \nabla_{\tan} g \|_{L^2(\partial\Omega)}
  \right\},
 \endaligned
\end{equation}
where we have used \eqref{gamma-e}.
By the energy estimates for the Stokes equations in $\Omega$, 
$$
\aligned
\|\nabla H_1 \|_{L^2(\Omega)} 
  & \le C \left\{ \| h \|_{H^{1/2}(\partial\Omega)} +|\gamma |\right\}\\
  & \le C\left\{  \| h\|_{L^2(\partial\Omega)}^{1/2} \| h \|^{1/2}_{H^1(\partial\Omega)} + |\gamma| \right\}\\
  & \le C \left\{ \e^{-1/2} \| h\|_{L^2(\partial\Omega)}
  + \e^{1/2} \| \nabla_{\tan} h \|_{L^2(\partial\Omega)} + |\gamma|  \right\}.
  \endaligned
$$
It follows that
\begin{equation}\label{5.2-3}
\e \|\nabla H_1 \|_{L^2(\Omega)}
\le C\sqrt{\e}
\left\{ \| h \|_{L^2(\partial\Omega)} + \e \|\nabla_{\tan} h \|_{L^2(\partial\Omega)}  + \e \| \phi_\e \nabla_{\tan} g \|_{L^2(\partial\Omega)}
 \right\}.
\end{equation}

To bound $\|H_1\|_{L^2(\Omega)}$, we use the following nontangential-maximal-function estimate,
\begin{equation}\label{max-e}
\| (H_1)^* \|_{L^2(\partial\Omega)}
\le C \| H_1 \|_{L^2(\partial\Omega)},
\end{equation}
where the nontangential maximal function $(H_1)^*$ on $\partial\Omega$ is defined by
$$
(H_1)^*(x)=\sup \big\{ |H_1(y)|: \ y\in \Omega \text{ and }   |y-x|< C_0\,  \text{dist}(y, \partial\Omega) \big\}
$$
for $x\in \partial\Omega$.
The estimate \eqref{max-e} was proved in \cite{FKV} for a bounded Lipschitz domain $\Omega$.
Let
$$
N_r =\big\{ x\in \Omega: \ \text{dist}(x, \partial \Omega)<   r  \big\}.
$$
It follows from \eqref{max-e} that
\begin{equation}\label{max-2}
\aligned
\| H_1 \|_{L^2(N_\e)}  & \le C \sqrt{\e} \| (H_1)^*\|_{L^2(\partial\Omega)}\\
& \le C \sqrt{\e} \left\{ \| h \|_{L^2(\partial \Omega)} +\e \| \phi_\e \nabla_{\tan} g \|_{L^2(\partial\Omega)}  \right\}.
\endaligned
\end{equation}

It remains to bound $\|H_1 \|_{L^2(\Omega\setminus N_\e)}$.
To this end, we consider the Dirichlet problem,
$$
\left\{
\aligned
-\Delta G +\nabla \pi  & =F & \quad & \text{ in } \Omega,\\
\text{\rm div} (G) & =0 & \quad & \text{ in } \Omega,\\
G & =0 & \quad & \text{ on } \partial\Omega,
\endaligned
\right.
$$
where $F \in C_0^\infty (\Omega \setminus  N_\e)$ and $\int_{ \Omega} \pi\, dx =0$.
Under the assumption that $\partial\Omega$ is of $C^{2, \alpha}$, we have the $W^{2, 2}$ estimates,

\begin{equation}\label{5.2-5}
\|G \|_{H^2(\Omega)}  +\|  \pi\|_{H^1(\Omega)} \le C \| F \|_{L^2(\Omega)}.
\end{equation}
This implies that
 \begin{equation}\label{5.6-6-1}
 \| \nabla G \|_{L^2(\partial\Omega)} + \| \pi \|_{L^2(\partial \Omega)}
 \le C \| F \|_{L^2(\Omega)}.
 \end{equation}
Moreover, since $F=0$ in $N_\e$, by covering $\partial\Omega$ with balls of radius $c\e$,
one may show that
\begin{equation}\label{5.2-6}
\aligned
\int_{\partial\Omega}
\left( |\nabla^2 G |^2 + |\nabla \pi |^2 \right)\, d\sigma
 & \le C \e^{-1} \|F \|_{L^2(\Omega)}^2.
 \endaligned
 \end{equation}
 To see this, we use the Green function representation for $G$ to obtain 
 \begin{equation}\label{5.5-6a}
 |\nabla^2 G(x)|
 \le C \int_{\Omega\setminus N_\e}  \frac{|F(y)|}{|x-y|^d } \, dy
 \end{equation}
 for $x\in \partial\Omega$.
 See e.g. \cite{G-Z-2019} for estimates of Green functions for the Stokes equations.
 Choose $\alpha, \beta \in (0, 1)$ such that $\alpha+\beta=1$,
 $\alpha >(1/2)$ and $\beta > (1/2)-(1/2d)$.
 It follows by the Cauchy inequality that for $x\in \partial\Omega$,
 $$
 \aligned
 |\nabla^2 G(x)|^2
 & \le C \left(\int_{\Omega\setminus N_\e} \frac{dy}{|x-y|^{2d\alpha}}\right)
\left(\int_{\Omega\setminus N_\e} \frac{|F(y)|^2}{|x-y|^{2d\beta}} \, dy \right)\\
&\le C \e^{d-2d\alpha}
\int_{\Omega\setminus N_\e} \frac{|F(y)|^2}{|x-y|^{2d\beta}} \, dy,
\endaligned
$$ 
where we have used the conditions $\alpha+\beta=1$ and $\alpha>(1/2)$.
Hence,
$$
\aligned
\int_{\partial\Omega}
|\nabla^2 G|^2\, d\sigma
&\le C \e^{d-2d\alpha}
\int_{\Omega\setminus N_\e}
|F(y)|^2\, dy
\sup_{y\in \Omega\setminus N_\e}
\int_{\partial\Omega} \frac{d\sigma (x)}{|x-y|^{2d\beta}}\\
& \le C \e^{-1}\int_{\Omega} |F(y)|^2\, dy,
\endaligned
$$
  where we have used the condition $\beta>(1/2)-(1/2d)$.
  This gives the estimate for $|\nabla^2 G|$ in \eqref{5.2-6}.
  The estimate for $\nabla \pi$ follows from the equation $-\Delta G+\nabla \pi=0$ near $\partial\Omega$.
   
 Finally, using integration by parts, we see that
 $$
 \aligned
  \int_\Omega H_1\cdot F\, dx 
 &= 
 \int_\Omega H_1 \cdot (-\Delta G + \nabla \pi)\, dx \\
 &=-
\int_{\partial \Omega}
H_1 \cdot \Big(\frac{\partial G}{\partial n} -n\pi \Big)\, d\sigma\\
&=-\int_{\partial\Omega} \Big( \e(  (\nabla_{\tan} \phi_\e) \cdot g ) n - \gamma (x-x_0)/d \Big)
 \cdot \Big(\frac{\partial G}{\partial n} -n\pi \Big)\, d\sigma.\\
\endaligned
$$
It follows  by using integration by parts on $\partial\Omega$ that 
$$
\aligned
& \Big|   \int_\Omega H_1 \cdot F\, dx  \Big|\\
&\le C\e  \int_{\partial\Omega}
|\phi_\e |\left( |\nabla g| |\nabla G| 
+|g| |\nabla^2 G|
+ |g| |\nabla G| 
+ |\nabla g| |\pi|
+ |g|  |\nabla\pi| + |g| |\pi | \right) \, d\sigma\\
& \qquad\qquad
+  |\gamma| 
\int_{\partial\Omega} \left( |\nabla G| + |\pi   |\right) \, d\sigma \\
&\le C \e \| \phi_\e g \|_{L^2(\partial\Omega)}
\left\{ \|\nabla^2 G \|_{L^2(\partial \Omega)} + \|\nabla G \|_{L^2(\partial\Omega)}
+ \|\nabla \pi \|_{L^2(\partial\Omega)} + \|\pi \|_{L^2(\partial\Omega)}  \right\}\\
&\qquad\qquad
+ C \e \| \phi_\e \nabla_{\tan} g \|_{L^2(\partial\Omega)} 
\left \{ \| \nabla G\|_{L^2(\partial\Omega)}
+ \| \pi \|_{L^2(\partial\Omega)}\right\},\\
\endaligned
$$
where we have used the Cauchy inequality and \eqref{gamma-e}.
This, together with  \eqref{5.6-6-1} and \eqref{5.2-6}, gives
$$
\Big|   \int_\Omega H_1 \cdot F\, dx  \Big|
\le C \e^{1/2}  \| F \|_{L^2(\Omega)}
\left\{  \| \phi_\e g \|_{L^2(\partial\Omega)}
+ \e^{1/2} \| \phi_\e \nabla_{\tan} g \|_{L^2(\partial\Omega)} \right\}.
$$
By duality we obtain 
\begin{equation}\label{5.6-8}
\| H_1 \|_{L^2(\Omega \setminus  N_\e)}
\le C \e^{1/2}  
\left\{  \| \phi_\e g \|_{L^2(\partial\Omega)}
+ \e^{1/2} \| \phi_\e \nabla_{\tan} g \|_{L^2(\partial\Omega)} \right\}.
\end{equation}
The desired estimate \eqref{5.2-0} 
follows from \eqref{5.2-1}, \eqref{5.2-3}, \eqref{max-2} and \eqref{5.6-8}.
\end{proof}

\begin{proof} [Proof of Theorem \ref{thm-Phi-2}]

Clearly,  by its definition, $\Phi_\e^{(2)} \in H^1(\Omega_\e; \R^d)$, $\Phi_\e^{(2)}=0$ on $\Gamma_\e$, and $\Phi_\e^{(2)}+V_\e =0$ on $\partial\Omega$.
Using the fact that $n \cdot K^\ell (f-\nabla P_0) =0$ on $\partial\Omega \cap \partial \Omega^\ell$, we obtain 
\begin{equation}\label{5.7-6}
\aligned
n \cdot h   
 &=-n \cdot W^\ell (x/\e) (f-\nabla P_0)\\
& =-n \cdot (W^\ell (x/\e) -K^\ell) (f-\nabla P_0)\\
& = -\frac{\e}{2} \left( 
n_i \frac{\partial}{\partial x_k} - n_k\frac{\partial}{\partial x_i}\right)
\left(\phi^\ell_{kij} (x/\e) \right)
\left(  f_j -\frac{\partial P_0}{\partial x_j} \right)
\endaligned
\end{equation}
on $\partial\Omega\cap \partial \Omega^\ell$.
It follows that
$$
\Big|
\int_{\partial\Omega} n \cdot h \, d\sigma \Big|
\le C\e \| \nabla (f-\nabla P_0) \|_{L^\infty(\partial \Omega)}.
$$
Hence, 
$$
\aligned
\| \text{\rm div}(\Phi_\e^{(2))} \|_{L^2(\Omega_\e)}
 & \le C |\gamma|
\le C \e \| \nabla ( f-\nabla P_0)\|_{L^\infty(\partial\Omega)}\\
& \le C \e \| f\|_{C^{1, 1/2} (\Omega)}.
\endaligned
$$
Finally,  in view of \eqref{5.7-6},  we apply  Theorem \ref{thm-N} to obtain 
$$
\aligned
\e \|\nabla \Phi_\e^{(2)} \|_{L^2(\Omega)}
&\le C \e^{1/2}
\left\{
\| f-\nabla P_0\|_{L^\infty(\partial\Omega)}
+ \e^{1/2} \| \nabla (f-\nabla P_0) \|_{L^\infty(\partial\Omega)} \right\}\\
& \le C \e^{1/2} \| f \|_{C^{1, 1/2} (\Omega)}.
\endaligned
$$
\end{proof}


\section{Interface correctors}\label{section-6}

In this section we construct a corrector $\Phi_\e^{(3)}$ for the interface $\Sigma$ and thus completes the proof of Theorem \ref{main-thm-2}.
Let $D=\Omega^\ell$ and $D_\e=\Omega^\ell_\e$ for some $1\le \ell \le L$.
Assume that $\partial D$ has no intersection with the boundary of the unbounded connected component of
 $\R^d\setminus \overline{\Omega}$.
Consider the Dirichlet problem, 
\begin{equation}
\label{D-6}
\left\{
\aligned
-\Delta u_\e +\nabla p_\e & =0 & \quad  & \text{ in } D_\e,\\
\text{\rm div} (u_\e) & =\gamma & \quad & \text{ in } D_\e,\\
u_\e  & =0 & \quad & \text{ on }  \Gamma_\e^\ell, \\
u_\e & = h & \quad & \text{ on } \partial  D,
\endaligned
\right.
\end{equation}
where $\Gamma_\e^\ell=\Gamma_\e \cap D$ and 
$$
\gamma =\frac{1}{|D_\e|}\int_{\partial D} h \cdot n \,  d\sigma.
$$
Let $W^+(y)=W ^\ell (y)$.
Fix $1\le j\le d$.
The boundary data $h$ on $\partial D$ in \eqref{D-6} is given as follows.
Let $\partial D= \cup_{k=1}^{k_0} \Sigma^k$, where $\Sigma^k$ are the connected component of $\partial D$.
On each $\Sigma^k$,  either 
\begin{equation}\label{6-1a}
h=0 \quad 
\end{equation}
or
\begin{equation}\label{6-1}
h=W^-_j (x/\e) -W^+_j (x/\e)
-W^-_{i}  (x/\e)  (K^-_{mj } -K^+_{mj})
\frac{ n_i n_m}{\langle nK^-, n\rangle },
\end{equation}
where $W^-(y)$ denotes the 1-periodic matrix defined by \eqref{cor-1} for the subdomain 
on the other side of $\Sigma^k$, and 
$$
K^+ =\int_Y W^+  (y) \, dy, \qquad K^- =\int_Y W^- (y)\, dy.
$$
In particular, if $\Sigma^ k\subset \partial \Omega$, we let $h=0$ on $\Sigma^k$.
Note that the  repeated indices $i, m$  in \eqref{6-1} are summed from $1$ to $d$.

\begin{lemma}\label{lemma-6-0}
Let $D$ be a bounded $C^{2, \alpha}$ domain in $\R^d$, $d\ge 2$.
Let $(u_\e, p_\e)$ be a weak solution of \eqref{D-6} with $\int_{D_\e} p_\e\, dx =0$, where $h$ is given by \eqref{6-1a}-\eqref{6-1}.
Then
\begin{equation}\label{6-1-0}
\e \|\nabla u_\e\|_{L^2(D_\e)}
+ \| u_\e\|_{L^2(D_\e)}
+ \|p_\e \|_{L^2(D_\e)}
\le C \sqrt{\e},
\end{equation}
and
\begin{equation}\label{6-1-00}
\| \text{\rm div}(u_\e)\|_{L^2(D_\e)} \le C \e.
\end{equation}
\end{lemma}

\begin{proof}

We apply Theorem \ref{thm-N}  with $\Omega=D$ to establish  \eqref{6-1-0}.
First, observe that  by \eqref{cell-1},
\begin{equation}
\| h \|_{L^2(\partial D)} + \e \|\nabla_{\tan} h \|_{L^2(\partial D)} \le C.
\end{equation}
Next, we compute $u\cdot n$ on $\Sigma^k$, assuming $h$ is given 
by \eqref{6-1}.
Note that
\begin{equation}\label{6-1-3}
\aligned
h\cdot n 
& = n_t W_{tj}^- (x/\e)-n_t W^+_{tj} (x/\e)
-n_t W^-_{ti} (x/\e) (K_{mj}^- -K_{mj}^+) \frac{n_i n_m}{\langle  n K^-, n \rangle}\\
&= n_t \left( W_{tj}^- (x/\e) -K_{tj}^-\right)
-n_t \left( W^+_{tj} (x/\e) -K_{tj}^+ \right)\\
&\qquad\qquad
-n_t \left( W^-_{ti}(x/\e)  - K_{ti}^-\right) (K_{mj}^- -K_{mj}^+) \frac{n_i n_m}{\langle  n K^-, n \rangle},
\endaligned
\end{equation}
where the repeated indices $t, i, m$ are summed from $1$ to $d$.
We use Lemma \ref{skew-lemma}  to write
\begin{equation}\label{6-1-5}
n_t \left(W_{ti}^\pm(x/\e) -K_{ti}^\pm \right)
=\frac{\e}{2}
\left( n_t \frac{\partial}{\partial x_s} -n_s\frac{\partial}{\partial x_t} \right) \left( \phi^\pm_{st i} (x/\e) \right).
\end{equation}
As a result, the function in the right-hand side of \eqref{6-1-3} may be written in the form
$\e (\nabla_{\tan} \phi_\e) \cdot g$ with  $(\phi_\e, g)$ satisfying 
$$
 \| \phi_\e \|_{L^2(\partial D)} + \| g\|_\infty + \|\nabla_{\tan}  g \|_\infty \le C.
$$
Consequently, the estimate \eqref{6-1-0}  follows from \eqref{5.2-0} in Theorem \ref{thm-N}.
Finally, note that \eqref{6-1-3} and \eqref{6-1-5} yield 
$$
\aligned
 \| \text{\rm div} (u_\e) \|_{L^2(D_\e)}
  & \le C \Big| \int_{\partial D} h\cdot n \, d\sigma \Big|\\
 & \le C \e.
\endaligned
$$
\end{proof}

Define
\begin{equation}\label{Phi-e-3}
\Phi^{(3)}_\e  = \sum_{\ell=1}^L I_\e^\ell (x)   (f-\nabla P_0) \chi_{\Omega_\e^\ell} \quad \text{ in } \Omega_\e,
\end{equation}
where  $I^\ell_\e =(I_{\e, 1} ^\ell , \dots, I_{\e, d}^\ell )$
   is a solution of \eqref{D-6} in $D_\e =\Omega^\ell_\e$ with $h$  given by \eqref{6-1a}-\eqref{6-1}.
To fix the boundary value $h$ for each subdomain, we assume that
 the unbounded connected component of $\R^d\setminus \overline{\Omega}$ shares boundary 
with $\Omega^1$,
and let $h=0$ on $\partial \Omega^1$. Thus, 
$I^1_\e (x) = 0$  and $\Phi_\e^{(3)} =0$ in $\Omega^1$.
Next, for each subdomain $\Omega^\ell$ that shares boundaries with $\partial\Omega^1$, we use  the boundary data \eqref{6-1} for 
the common boundary with $\partial \Omega^1$ and let $h=0$ on other components of $\partial\Omega^\ell$.
We continue this process.
More precisely, at each step, we use \eqref{6-1} on the connected component  $\Sigma^k$ 
of $\partial \Omega^\ell$ if  $\Sigma^k$ is also the connected component of  the boundary of a subdomain 
considered in the previous step, and let $h=0$ on the remaining components.
We point out that at each interface $\Sigma^k$,  the nonzero data \eqref{6-1} is used only once.
Also, $h=0$ on $\partial\Omega$.

\begin{lemma}\label{lemma-6-1}
Let $\Phi^{(3)} _\e$ be given by \eqref{Phi-e-3} with $f\in C^{1, 1/2} ({\Omega; \R^d})$.
Then $V_\e +\Phi_\e^{(3)} \in H^1(\Omega_\e; \R^d)$.
\end{lemma}

\begin{proof}
Let $\Psi_\e=V_\e + \Phi_\e^{(3)}$.
Since $f\in C^{1, 1/2}(\Omega) $ implies that $\nabla^2 P_0$ is bounded in each subdomain,
it follows that $\Psi_\e \in H^1(\Omega_\e^\ell; \R^d)$ for  $1\le \ell \le L$.
Thus, to show $\Psi_\e \in H^1(\Omega_\e; \R^d)$, it suffices to show that the trace of  $\Psi_\e$ is continuous 
across each interface $\Sigma^k$.

Suppose  that $\Sigma^k$ is the common boundary of subdomains $\Omega^+$ and $\Omega^-$.
Let $\Psi_\e^\pm $ denote the trace of $\Psi_\e$ on $\Sigma^k$, taken from $\Omega^\pm$ respectively.
Recall that in  the definition of $\{ I_\e^\ell \}$, the non-zero data \eqref{6-1} is used once on each interface. Assume that
the non-zero data on $\Sigma^k$ is used for $\Omega^+$. Then
$$
\Psi_\e^+ -\Psi_\e^-
=\left( W^+ (x/\e)  + I^+ _\e (x)\right) (f-\nabla P_0)^+
-W^- (x/\e) (f-\nabla P_0)^-,
$$
where $I^+_\e$ is given by \eqref{6-1}. 
It follows that
$$
\aligned
\Psi_\e^+ -\Psi_\e^-
&=\left( W_j^- (x/\e) -W_i^-(x/\e) ( K_{mj}^- -K_{mj}^+)
\frac{n_i n_m}{\langle n K^-, n \rangle} \right) \left( f_j -\frac{\partial P_0}{\partial x_j} \right)^+\\
& \qquad\qquad
-W^-_j (x/\e) \left( f_j -\frac{\partial P_0}{\partial x_j} \right)^-\\
&=W_j^- (x/\e)
\left\{
\left(\frac{\partial P_0}{\partial x_j} \right)^-
-\left( \frac{\partial P_0}{\partial x_j} \right)^+
-\frac{n_j n_m}{\langle nK^-, n\rangle}
K_{mi}^- \left( f_i -\frac{\partial P_0}{\partial x_i} \right)^+\right\}\\
& \qquad\qquad+W_j^-(x/\e) 
\frac{n_j n_m}{\langle n K^-, n \rangle} K_{mi}^- \left( f_i -\frac{\partial P_0}{ \partial x_i} \right)^- ,
\endaligned
$$
where we have used the observation that 
\begin{equation}\label{interface-1}
n_m K_{mi}^+ \left( f_i -\frac{\partial P_0}{\partial x_i} \right)^+
=n_m K_{mi}^-  \left( f_i -\frac{\partial P_0}{\partial x_i} \right)^-
\end{equation}
on the interface. Thus,
$$
\aligned
\Psi_\e^+ -\Psi_\e^-
 & =W_j^-(x/\e)
\left\{
\left(\frac{\partial P_0}{\partial x_j} \right)^-
-\left(\frac{\partial P_0}{\partial x_j}\right)^+
-\frac{n_j n_m}{\langle n K^-, n \rangle}
K_{mi}^- \left(
\left(\frac{\partial P_0}{\partial x_i} \right)^-
-\left(\frac{\partial P_0}{\partial x_i} \right)^+ \right)\right\}\\
&=W_j^-(x/\e)
\left\{
\delta_{ij} 
-\frac{n_j n_m}{\langle n K^-, n \rangle}
K_{mi}^- \right\} \left(
\left(\frac{\partial P_0}{\partial x_i} \right)^-
-\left(\frac{\partial P_0}{\partial x_i} \right)^+ \right).
\endaligned
$$
Since 
$$
n_i \left\{
\delta_{ij} 
-\frac{n_j n_m}{\langle n K^-, n \rangle}
K_{mi}^- \right\}=0
$$
and $(\nabla_{\tan} P_0)^+ = (\nabla_{\tan} P_0)^-$ on $\Sigma^k$, we obtain $\Psi_\e^+ =\Psi_\e^-$ on $\Sigma^k$.
\end{proof}

\begin{thm}\label{thm-PP3}
Let $\Phi_\e^{(3)}$ be defined by \eqref{Phi-e-3} with $f\in C^{1, 1/2} (\Omega; \R^d)$.
Then $V_\e + \Phi_\e^{(3)}\in H^1(\Omega; \R^d)$ and 
\begin{equation}
\e \|\nabla \Phi_\e^{(3)} \|_{L^2(\Omega^\ell_\e)}
+\| \text{\rm div} (\Phi_\e^{(3)} ) \|_{L^2(\Omega^\ell_\e)}
 \le C \e^{1/2} \| f \|_{C^{1, 1/2}(\Omega)}
\end{equation}
for $1\le \ell \le L$.
\end{thm}

\begin{proof} 
By Lemma \ref{lemma-6-1}, we have $V_\e+ \Phi_\e^{(3)} \in H^1(\Omega; \R^d)$.
Note that by Lemma \ref{lemma-6-0},
$$
\e \| \nabla I_\e^\ell \|_{L^2(\Omega_\e^\ell)} + \| I_\e^\ell \|_{L^2(\Omega_\e^\ell)} 
+ \| \text{\rm div} (I_\e^\ell)\|_{L^2(\Omega_\e^\ell)}
\le C \e^{1/2}
$$
for $1\le \ell \le L$.
It follows that 
$$
\aligned
\e \| \nabla \Phi_\e^{(3)} \|_{L^2(\Omega_\e^\ell)}
& \le \e \|\nabla I_\e^\ell \|_{L^2(\Omega_\e^\ell)}  \| f -\nabla P_0\|_{L^\infty(\Omega_\e^\ell)}
+  \e \| I_\e^\ell \|_{L^2(\Omega_\e^\ell)} \| \nabla (f-\nabla P_0) \|_{L^\infty(\Omega_\e^\ell)}\\
& \le
 C \e^{1/2} \| f\|_{C^{1, 1/2}(\Omega)},
\endaligned
$$
and
$$
\aligned
\|\text{\rm div} (\Phi_\e^{(3)} ) \|_{L^2(\Omega_\e^\ell)}
&\le \| \text{\rm div}(I_\e^\ell) \|_{L^2(\Omega^\ell_\e)} \| f-\nabla P_0 \|_{L^\infty(\Omega^\ell)}
+ \| I^\ell_\e \|_{L^2(\Omega_\e^\ell)} \|\nabla (f-\nabla P_0) \|_{L^\infty(\Omega_\e^\ell)}\\
&\le
 C \e^{1/2} \| f \|_{C^{1, 1/2} (\Omega)}.
\endaligned
$$
\end{proof}


 \bibliographystyle{amsplain}
 
\bibliography{Darcy-two.bbl}

\providecommand{\bysame}{\leavevmode\hbox to3em{\hrulefill}\thinspace}
\providecommand{\MR}{\relax\ifhmode\unskip\space\fi MR }
\providecommand{\MRhref}[2]{%
  \href{http://www.ams.org/mathscinet-getitem?mr=#1}{#2}
}
\providecommand{\href}[2]{#2}
\begin{thebibliography}{10}

\bibitem{Allaire-89}
G.~Allaire, \emph{Homogenization of the {S}tokes flow in a connected porous
  medium}, Asymptotic Anal. \textbf{2} (1989), no.~3, 203--222.

\bibitem{Allaire-1991}
\bysame, \emph{Homogenization of the {N}avier-{S}tokes equations in open sets
  perforated with tiny holes. {I}. {A}bstract framework, a volume distribution
  of holes}, Arch. Rational Mech. Anal. \textbf{113} (1990), no.~3, 209--259.

\bibitem{Allaire-91a}
\bysame, \emph{Continuity of the {D}arcy's law in the low-volume fraction
  limit}, Ann. Scuola Norm. Sup. Pisa Cl. Sci. (4) \textbf{18} (1991), no.~4,
  475--499.

\bibitem{Allaire-1997}
G.~Allaire and A.~Mikeli\'{c}, \emph{One-phase {N}ewtonian flow},
  Homogenization and porous media, Interdiscip. Appl. Math., vol.~6, Springer,
  New York, 1997, pp.~45--76, 259--275.

\bibitem{B}
M.~Belhadj, E.~Canc\`es, J.-F. Gerbeau, and A.~Mikeli\'{c},
  \emph{Homogenization approach to filtration through a fibrous medium}, Netw.
  Heterog. Media \textbf{2} (2007), no.~3, 529--550.

\bibitem{Dagan}
G.~Dagan, \emph{{Flow and Transport in Porous Formations}}, Springer, 1989.

\bibitem{FKV}
E.~B. Fabes, C.~E. Kenig, and G.~C. Verchota, \emph{The {D}irichlet problem for
  the {S}tokes system on {L}ipschitz domains}, Duke Math. J. \textbf{57}
  (1988), no.~3, 769--793.

\bibitem{G-Z-2019}
S.~Gu and J.~Zhuge, \emph{Periodic homogenization of {G}reen's functions for
  {S}tokes systems}, Calc. Var. Partial Differential Equations \textbf{58}
  (2019), no.~3, Paper No. 114, 46.

\bibitem{Mikelic-two}
W.~J\"{a}ger and A.~Mikeli\'{c}, \emph{On the boundary conditions at the
  contact interface between two porous media}, Partial differential equations
  ({P}raha, 1998), Chapman \& Hall/CRC Res. Notes Math., vol. 406, Chapman \&
  Hall/CRC, Boca Raton, FL, 2000, pp.~175--186.

\bibitem{LA-1990}
R.~Lipton and M.~Avellaneda, \emph{Darcy's law for slow viscous flow past a
  stationary array of bubbles}, Proc. Roy. Soc. Edinburgh Sect. A \textbf{114}
  (1990), no.~1-2, 71--79.

\bibitem{Mik-1996}
E.~Maru\v{s}i\'{c}-Paloka and A.~Mikeli\'{c}, \emph{An error estimate for
  correctors in the homogenization of the {S}tokes and the {N}avier-{S}tokes
  equations in a porous medium}, Boll. Un. Mat. Ital. A (7) \textbf{10} (1996),
  no.~3, 661--671.

\bibitem{MGG}
A.~M. Meirmanov, O.~V. Galtsev, and S.~A. Gritsenko, \emph{On homogenized
  equations of filtration in two domains with a common boundary}, Izv. Ross.
  Akad. Nauk Ser. Mat. \textbf{83} (2019), no.~2, 142--173.

\bibitem{Sanchez-1980}
E.~S\'{a}nchez-Palencia, \emph{Nonhomogeneous media and vibration theory},
  Lecture Notes in Physics, vol. 127, Springer-Verlag, Berlin-New York, 1980.

\bibitem{Shen-Darcy-2}
Z.~Shen, \emph{Compactness and large-scale regularity for {D}arcy's law}, J.
  Math. Pures Appl. (9) \textbf{163} (2022), 673--701.

\bibitem{Shen-Darcy-1}
\bysame, \emph{Sharp convergence rates for {D}arcy's law}, Comm. Partial
  Differential Equations \textbf{47} (2022), no.~6, 1098--1123.

\bibitem{Zhuge}
J.~Zhuge, \emph{Regularity of a transmission problem and periodic
  homogenization}, J. Math. Pures Appl. (9) \textbf{153} (2021), 213--247.

\end{thebibliography}

\bigskip

\begin{flushleft}

\bigskip

Zhongwei Shen,
Department of Mathematics,
University of Kentucky,
Lexington, Kentucky 40506,
USA.

E-mail: zshen2@uky.edu
\end{flushleft}

\bigskip

\medskip

\end{document}